\documentclass[10pt]{amsart}
\usepackage{pifont}
\usepackage{amssymb}
\usepackage{amsmath}
\usepackage{graphicx}

\newtheorem{thm}{Theorem}
\newtheorem{lem}[thm]{Lemma}
\newtheorem{cor}[thm]{Corollary}
\newtheorem{prop}[thm]{Proposition}
\newtheorem{rem}{Remark}
\newtheorem{defn}{Definition}
\newtheorem{claim}{Claim}

\title[Asymptotics for Carleman polynomials]{Asymptotic behavior and zero distribution of Carleman orthogonal polynomials}

\author{Peter Dragnev}

\address{Indiana-Purdue University Fort Wayne, Department of Mathematical Sciences,\\
2101 E. Coliseum Blvd., Fort Wayne, IN 46805-1499, USA}

\email{dragnevp@ipfw.edu}

\author{Erwin Mi\~{n}a-D\'{\i}az}

\address{The University of Mississippi,  Department of Mathematics, P. O. Box 1848, \\ University, MS  38677-1848, USA}
\email{minadiaz@olemiss.edu}

\begin{document}
\begin{abstract}
 Let $L$ be an analytic Jordan curve and let $\{p_n(z)\}_{n=0}^\infty$ be the sequence of polynomials that are orthonormal with respect to the area measure over the interior of $L$. A well-known result of Carleman states that
\begin{equation}\label{eq12}
 \lim_{n\to\infty}\frac{p_n(z)}{\sqrt{(n+1)/\pi}\,[\phi(z)]^{n}}= \phi'(z)
\end{equation}
locally uniformly on certain open neighborhood of the closed exterior of $L$, where $\phi$ is the canonical conformal map of the exterior of $L$ onto the exterior of the unit circle. In this paper we extend the validity of (\ref{eq12}) to a maximal open set, every boundary point of which is an accumulation point of the zeros of the $p_n$'s. Some consequences on the limiting distribution of the zeros are discussed, and the results are illustrated with two concrete examples and numerical computations.
 \end{abstract}
 
 \maketitle

\section{Introduction}

Polynomials of a complex variable that are orthogonal over a bounded domain of the complex plane were first investigated by T. Carleman \cite{Carleman} in 1922, and considerable progress has been made since then in clarifying questions such as the convergence of Fourier series in these polynomials, their completeness in different Banach spaces of analytic functions,  their asymptotic behavior and more recently the limiting distribution of their zeros (see, e.g., the monograph by  Suetin \cite{Suetin1} and the references therein, together with the papers \cite{Andriev1}, \cite{Andriev2}, \cite{Eiermann}, \cite{Levin}, \cite{victor}, \cite{mina1}, \cite{mina2}).  Generally speaking, all these questions are dependent of the boundary properties of the orthogonality domain, and in the present paper we specifically consider the case of a domain with analytic boundary, having as the subject of our investigation the asymptotic behavior and zero distribution of the corresponding orthogonal polynomials.

Let $L_1$ be an analytic Jordan curve in the
complex plane $\mathbb{C}$ and let $G_1$ be its interior domain, that is, the bounded component of $\mathbb{C}\setminus L_1$. By applying the Gram-Schmidt orthonormalization process to the sequence $1,z,z^2,\ldots$, we can construct a unique sequence of complex analytic polynomials $\{p_n(z)\}_{n=0}^\infty$ (each $p_n$ having degree $n$ and positive leading coefficient) that are orthonormal over $G_1$ with respect to the normalized area measure $\pi^{-1}dxdy$, that is, satisfying
\begin{equation}\label{eq64}
\frac{1}{\pi}\int_{G_1}p_{n}(z)\overline{p_{m}(z)}dxdy=\left\{\begin{array}{ll}
                                                                 0,\ &\ n\not=m, \\
                                                                 1,\ &\ n=m.
                                                               \end{array}
\right.
\end{equation}

These polynomials were first examined by T. Carleman in his study \cite{Carleman} on the approximation of analytic functions by polynomials over a bounded Jordan domain. In particular, Carleman investigated the behavior of $p_n(z)$ as $n\to\infty$, finding a fundamental result that we state below after setting some needed notation.

For  a planar set $K$ and a function $f$ defined on
$K$, $\overline{K}$ and $\partial K$ denote, respectively, the closure and the boundary of $K$ in the extended complex plane  $\overline{\mathbb{C}}$, and
$f(K):=\{f(z):z\in K\}$.

Given $r\geq 0$, we set
\[
\mathbb{T}_r:=\{w:|w|=r\},\quad
\Delta_r:=\{w:r<|w|\leq \infty\},\quad
\mathbb{D}_r:=\{w:|w|<r \}.
\]

Let $\Omega_1$ be the unbounded component of $\overline{\mathbb{C}}\setminus L_1$, and let $\psi(w)$ be the unique conformal map of
$\Delta_1$ onto $\Omega_1$ satisfying that
$\psi(\infty)=\infty$, $\psi'(\infty)>0$. Let
$\rho\geq 0$ be the radius of univalency of $\psi$, that is, the smallest number such that
$\psi$ has an analytic and \emph{univalent}
continuation to $\{w:\rho<|w|<\infty\}$. Because $L_1$ is an analytic Jordan curve,
$\rho<1$. For every $\rho\leq r<\infty$, set
\begin{equation}\label{eq30}
\Omega_r:=\psi(\Delta_r),\quad L_r:=\partial
\Omega_r,\quad G_r:=\mathbb{C}\setminus
\overline{\Omega}_r,
\end{equation}
and let
\[
\phi(z):\Omega_{\rho}\to
\Delta_{\rho}
\]
be the inverse of $\psi$. Observe that for $r>\rho$, $L_r$ is an analytic Jordan curve.

Carleman's fundamental result  mentioned above  (\cite[Satz IV]{Carleman}, see also \cite[Sec. 2]{Gaier1}) states that
\begin{equation}\label{eq8}
h_n(z):=\frac{p_{n}(z)}{\sqrt{n+1}[\phi(z)]^n}-\phi'(z)=o(1)
\end{equation}
locally uniformly on $\Omega_\rho$ as $n\to\infty$. More precisely, Carleman proved that $h_n(z)$ converges uniformly as $n\to\infty$ to zero on each $\overline{\Omega}_r$, $r>\rho$, with the rate
\begin{equation}\label{eqq38}
h_n(z)=\left\{\begin{array}{ll}
O\left(\sqrt{n}\rho^n\right),\ &\ r\geq 1, \\
          {\displaystyle O\left(n^{-1/2}(\rho/r)^n\right)},\ &\ \rho<r<1.
       \end{array}\right.
\end{equation}

Progress in understanding the behavior of $p_n$ in $\mathbb{C}\setminus \Omega_\rho$ has been recently made in \cite{mina1}, where the following asymptotic integral representation for $p_n$ has been obtained. If $\varphi(z)$ is a conformal map of $G_1$ onto the unit disk, then $\varphi$ has an analytic continuation across $L_1$, so that the composition $\varphi(\psi(w))$ is well-defined and analytic on the unit circle $\mathbb{T}_1$, and we have (see \cite[Theorem 2.1.2 and Eq. (14)]{mina1})
\emph{\begin{equation}\label{eq2}
 p_{n}(z)=\frac{\sqrt{n+1}\,\varphi'(z)}{2\pi
 i}\oint_{\mathbb{T}_1}\frac{w^{n}dw}{\varphi(\psi(w))-\varphi(z)}
 +\epsilon_n(z),\quad z\in G_1,\quad n\geq 0,
\end{equation}
where the functions $\epsilon_n(z)$
are analytic in $G_{1/\rho}$
and have the following property: if $E\subset
G_{1/\rho}$ is such that for some
$0<\tau<1/\rho$,
\begin{equation*}
p_n(z)=O\left(\sqrt{n}\tau^n\right)
\end{equation*}
uniformly on $E$ as $n\to\infty$, then
\begin{equation*}
\epsilon_n(z)=O\left(\sqrt{n}(\tau\rho)^n\right)
\end{equation*}
uniformly on $E$ as $n\to\infty$.
}

This representation is used in \cite{mina1} to derive finer asymptotics for $p_n$ and its zeros under the assumption that (roughly speaking) the boundary of $\Omega_\rho$ is a piecewise analytic curve. As a little bonus, one also obtains from (\ref{eq2}) (see \cite[Corollary 2.1.3]{mina1}) that the  $\sqrt{n}$ factor occurring in (\ref{eqq38}) for the case $r\geq 1$ can be dropped.

In the present paper we exploit (\ref{eq2}) to extend the validity of Carleman's formula (\ref{eq8}) from the band $\Omega_{\rho}\cap G_1$ toward a maximal open subset $\Sigma_1$ of $G_1$ that is, in general, larger than  $\Omega_{\rho}\cap G_1$. $\Sigma_1$ is the largest open subset of $G_1$ where a strong asymptotic formula like (\ref{eq8}) holds true, and every point of $\partial\Sigma_1\cap G_1$ is an accumulation point of the zeros of the $p_n$'s.

These results are stated  in Section \ref{sec2} as Theorems \ref{thm3} and \ref{thm4}. Some consequences on the limiting distribution of the normalized counting measures of the zeros of the $p_n$'s are presented as Theorem \ref{thm2}. The definition of $\Sigma_1$ and its finding in concrete situations involves the meromorphic continuation of the map $\varphi(\psi(w))$ occurring inside the integrand in (\ref{eq2}). We study such a continuation in Propositions \ref{prop1}, \ref{prop6} and \ref{prop2} of Section \ref{sec2}. In Section \ref{sec3}, we discuss two concrete examples to illustrate the main results and the use of the propositions. Numerical computations of the zeros of the orthogonal polynomials are provided in both examples, and a problem concerning the behavior of the polynomials in the second example is posed for future investigation. Finally, the results are proven in Sections \ref{sec4} and \ref{sec5}.

\section{Main results}\label{sec2}
Let $\varphi$ be a conformal map of $G_1$ onto $\mathbb{D}_1$.
Because $L_1$ is a Jordan curve, $\varphi$ can be extended as a continuous and
bijective function $\varphi:\overline{G}_1\to
\overline{\mathbb{D}}_1$. Moreover, being $L_1$
analytic, $\varphi$ has a one-to-one meromorphic
continuation to $G_{1/\rho}$, which satisfies
\begin{equation}\label{eq66}
\varphi(z)=\frac{1}{\overline{\varphi\left(z^*\right)}},
\quad z\in  \Omega_{\rho}\cap G_{1/\rho},
\end{equation}
where
\begin{equation}\label{eq70}
z^*:=\psi\left(1\big/\overline{\phi(z)}\right)
\end{equation}
is the Schwarz reflection  about $L_1$ of the
point $z\in \Omega_{\rho}\cap G_{1/\rho}$ (see
\cite{Davis} for details).

The function $\psi$ is analytic and univalent on $\Delta_\rho$, mapping the annulus $\rho<|w|<1/\rho$ conformally onto the band $\Omega_\rho\cap G_{1/\rho}$, so that
\[
\varphi(\psi(w)), \quad \rho<|w|<1/\rho,
\]
is a one-to-one meromorphic function that is analytic on $\rho<|w|\leq 1$.

\begin{defn}\label{defn1}Let $\mu\geq 0$ be the smallest number such that $\varphi(\psi(w))$ has a \emph{meromorphic} continuation, denoted by $h_\varphi(w)$, to the annulus \[\left\{w:\mu<|w|<1/\rho\right\}.
\]
Let $\Sigma$ be the set of those points $z\in G_1$ for which the equation
\begin{equation}\label{eq4}
h_\varphi(w)=\varphi(z)
\end{equation}
has at least one solution in $\mu<|w|<1$. Let $\Sigma_0:=G_1\setminus \Sigma$.
\end{defn}

We say that a solution $\omega$ of (\ref{eq4}) has multiplicity $\alpha\geq 1$ if
\begin{equation*}
h^{(\alpha)}_\varphi(\omega)\not=0, \quad h^{(j)}_\varphi(\omega)=0\quad 1\leq j<\alpha.
\end{equation*}
Consider a $z\in \Sigma$. Since $h_\varphi(w)$ is one-to-one on $\rho<|w|<1$, among the solutions to (\ref{eq4}), a finite number, say $\omega_{z,1},  \ldots, \omega_{z,s}$ ($s\geq 1$), will have largest modulus. Let $\alpha_{z,k}$ denote the multiplicity of  $\omega_{z,k}$ ($1\leq k\leq s$). We decompose $\Sigma$ in subsets $\Sigma_p$, $p=1,2,\ldots$, defined by the relation
\begin{equation}\label{eq102}
z\in \Sigma_p \Leftrightarrow \alpha_{z,1}+\cdots+ \alpha_{z,s}=p.
\end{equation}

Thus, $\Sigma_1$ consists of those points $z\in \Sigma$ such that the equation
(\ref{eq4}) has exactly one solution in $\mu<|w|<1$ of largest modulus, and this solution is simple.

Let the function $\Phi:\Sigma_1\to \{w:\mu<|w|<1\}$ be defined as
\[
\Phi(z):=\omega_{z,1},\quad z\in \Sigma_1,
\]
and let $r:G_1\to[\mu,1)$ be defined as
\begin{equation}\label{eq5}
r(z):=\left\{\begin{array}{ll}
               |\omega_{z,1}|, & z\in \Sigma, \\
               \mu,  & z\in \Sigma_0.
             \end{array}\right.
\end{equation}

It is easy to see (see the first two paragraphs of Section \ref{sec4}) that the number $\mu$, the sets $\Sigma$, $\Sigma_p$, and the functions $\Phi(z)$, $r(z)$ are, indeed, independent of the choice of the interior map $\varphi$. Also (see Corollary \ref{cor1} in Section \ref{sec4}) $\Sigma$ and $\Sigma_1$ are open, $\Sigma_1\supset \Omega_\rho\cap G_1$, the map $\Phi$ is a \emph{one-to-one} analytic function and $r(z)$ is continuous.

Note that
\[
\Phi(z)=\phi(z), \quad z\in \Omega_\rho\cap G_1,
\]
and that $r(z)=|\Phi(z)|$ for all $z\in \Sigma_1$. Our main result is the following theorem.

\begin{thm} \label{thm3}
\begin{enumerate} \item For every compact set $E\subset \Sigma_1$, there exists a number $0< \delta<1$ such that
\begin{equation*}
\frac{p_{n}(z)}{\sqrt{n+1}[\Phi(z)]^n}-\Phi'(z)=O(\delta^n)
\end{equation*}
uniformly on $E$ as $n\to\infty$. 
\item
\begin{equation*}
\limsup_{n\to\infty}|p_n(z)|^{1/n}=r(z),\quad z\in G_1.
\end{equation*}
\end{enumerate}
\end{thm}

This result has several implications on the asymptotic zero distribution of the orthogonal polynomials. Consider the set $\mathcal{Z}$ of accumulation points of the zeros of the $p_n$'s, that is, $\mathcal{Z}$ consists of those points $t\in\overline{\mathbb{C}}$ such that every neighborhood of $t$ contains zeros of infinitely many polynomials $p_n$. A simple consequence of (\ref{eq8}) and Theorem \ref{thm3}(a) is that every closed subset of $\overline{\Omega}_1\cup\Sigma_1$ may contain zeros of at most finitely many polynomials $p_n$, and therefore, $\mathcal{Z}\subset G_1\setminus \Sigma_1$. Moreover, we have

\begin{thm} \label{thm4} $\partial\Sigma_1\cap G_1\subset \mathcal{Z}$.
\end{thm}

Let now $z_{n,1},\ldots,z_{n,n}$ be the zeros of $p_n$, let $\delta_z$ denote the unit point mass at $z$, and let
\[
\nu_n:=\frac{1}{n}\sum_{j=1}^n\delta_{z_{n,j}}
\]
be the so-called normalized counting measure of the zeros of $p_n$. The sequence $\{\nu_n\}_{n=1}^\infty$ is said to converge in the weak*-topology to the measure $\nu$ if  \[\lim_{n\to\infty}\int_{\mathbb{C}} fd\nu_n=\int_\mathbb{C}fd\nu\] for every continuous function $f:\mathbb{C}\to\mathbb{C}$.

In preparation for the next theorem concerning the weak*-limit points of the sequence $\{\nu_n\}$, we recall that the logarithmic potential of a compactly supported measure $\nu$ is the superharmonic function
\begin{equation*}
U^\nu(z):=-\int_{\mathbb{C}}\log |t-z|d\nu(t), \quad z\in \overline{\mathbb{C}}.
\end{equation*}

Let us extend the function $r(z)$ in (\ref{eq5}) to the entire complex plane by setting $r(z)=|\phi(z)|$ for $z\in\mathbb{C}\setminus G_1$. Then, there exists a unique measure $\lambda$ having logarithmic potential
\begin{equation}\label{eqq200}
U^\lambda(z)=-\log r(z)+\log\phi'(\infty), \quad z\in\mathbb{C}.
\end{equation}
This $\lambda$ is a probability measure whose support coincides with $\partial \Sigma_1\cap G_1$, and we have
\begin{thm}\label{thm2} If the interior of $\Sigma_0$ is connected, then some subsequence of $\{\nu_n\}_{n=1}^\infty$ converges in the weak*-topology to $\lambda$, and this is true of the entire sequence $\{\nu_n\}_{n=1}^\infty$ if the interior of $\Sigma_0$ is empty.
\end{thm}

\begin{rem}\label{rem3}  Given that, by definition, $h_\varphi(w)=\varphi(\psi(w))$ for $\rho<|w|<1/\rho$, it is easy to verify that
\[
 \mu=\rho \Leftrightarrow G_1=\Sigma_1\cup\Sigma_0 \Leftrightarrow \Sigma_0=\mathbb{C}\setminus\Omega_\rho.
\]
Hence if  $G_1=\Sigma_1\cup\Sigma_0$ and $\Sigma_0=\{z_0\}$ is a singleton,  then $\rho=0$ and the set $G_1$ is an open disk centered at $z_0$, say $G_1=\{z:|z-z_0|<s\}$ for some $s>0$. In this case, 
\[
p_n(z)=\sqrt{n+1}\,s^{-n-1}(z-z_0)^n,\quad n\geq 0,
\]
as can be verified directly from the orthogonality relations (\ref{eq64}), so that $\nu_n=\delta_{z_0}=\lambda$, $n\geq 1$.

If $G_1=\Sigma_1\cup\Sigma_0$,  $\Sigma_0$ not a singleton, then $\rho>0$, $\mathrm{supp}(\lambda)=\partial \Sigma_1\cap G_1 =\partial \Sigma_0$ and by (\ref{eqq200}), $U^\lambda(z)$ is (a finite) constant on $\Sigma_0$. Hence $\lambda$ is the equilibrium distribution with respect to the logarithmic potential of the compact set $\Sigma_0$ (cf. Section III.2 and Theorem III.15 of \cite{Tsuji}) . 
\end{rem}

The proofs of Theorems \ref{thm4} and \ref{thm2} can be accomplished by using a series of arguments previously developed by Ullman \cite{Ullman2} and  Kuijlaars, Saff \cite{KS} in the context of Faber polynomials. These arguments are of a very general nature and can likewise be used in our setting without  any essential modification. We shall therefore provide only an outline of these proofs at the end of Section \ref{sec4}.

For concrete instances of a curve $L_1$, the difficulty of finding the corresponding number $\mu$ and set $\Sigma_1$ may be reduced with the use of the following three propositions. These establish some properties of the meromorphic continuation of the map $\varphi(\psi(w))$. Their use is illustrated in the examples of next section.

For a domain $\mathfrak{D}\subset\overline{\mathbb{C}}$, we denote by $\mathfrak{D}^*$ the reflection of $\mathfrak{D}$ about the unit circle, i.e.,
\[
\mathfrak{D}^*:=\{1/\overline{w}:w\in \mathfrak{D}\}.
\]

\begin{prop}\label{prop1} Let $\varphi$ be a conformal map of $G_1$ onto $\mathbb{D}_1$ and let us denote by the same letter $\varphi$ its meromorphic continuation to $G_{1/\rho}$. Let $\mathfrak{D}$ be a domain such that $\{w:\rho<|w|<1\}\subset\mathfrak{D}\subset \mathbb{D}_1$.

The function $\varphi(\psi(w))$, originally defined in $\rho<|w|<1/\rho$, has a meromorphic continuation to $\mathfrak{D}$, if and only if it has a meromorphic continuation to $\mathfrak{D}^*$, if and only if $\varphi(z)$ has a meromorphic continuation to $\overline{G}_1\cup\psi(\mathfrak{D}^*)$. If $h_\varphi(w)$ denotes the meromorphic continuation of $\varphi(\psi(w))$ to $\mathfrak{D}\cup\mathbb{T}_1\cup \mathfrak{D}^*$, then
\begin{equation}\label{eq32}
h_\varphi(w)=\frac{1}{\overline{h_\varphi(1/\overline{w})}}
\end{equation}
for all $w\in \mathfrak{D}\cup\mathbb{T}_1\cup \mathfrak{D}^*$. In particular, for $\mu$ as in Definition \ref{defn1}, we have
\begin{align*}
&\sup\left\{r\geq 1/\rho: \varphi(\psi(w))\ has\ a\ meromorphic\ continuation\ to\ \rho<|w|<r\right\}\\
&=\sup\left\{r\geq 1/\rho: \varphi\ has\ a\ meromorphic\ continuation\ to\ G_r \right\}\\
&=1/\mu.
\end{align*}
\end{prop}

The next proposition tells us that if $\mathfrak{D}\subset \mathbb{D}_1$ is a domain that can be exhausted by continuously expanding domains $\mathfrak{D}_t$, each satisfying that $\psi(\overline{\mathfrak{D}}_t)\subset \overline{G}_1\cup \psi(\mathfrak{D}^*_t)$, then $\varphi(\psi(w))$ has a meromorphic continuation to $\mathfrak{D}$. The precise formulation is as follows.

\begin{prop} \label{prop6} Let $\{\mathfrak{D}_t:a\leq t<  b\}$ be a collection of domains  such that for every $a\leq t_0<t_1<  b$,
\begin{equation}\label{eqq32}
\{w:\rho<|w|<1\}\subset \mathfrak{D}_{t_0}\subset \mathfrak{D}_{t_1}\subset \mathbb{D}_1,\quad \bigcap_{t>t_0}\mathfrak{D}_t = \overline{\mathfrak{D}}_{t_0}\setminus\mathbb{T}_1.
\end{equation}
Let $\mathfrak{D}:=\cup_{a\leq t< b}\mathfrak{D}_{t}$ and suppose that $\psi$ is meromorphic in $\mathfrak{D}$ and satisfies
\begin{equation}\label{eqq31}
\psi(\mathfrak{D}_a)\subset G_1, \quad
\psi(\overline{\mathfrak{D}}_t)\subset \overline{G}_1\cup \psi(\mathfrak{D}^*_t),\quad a< t<  b.
\end{equation}
Then, $\varphi(\psi(w))$ admits a meromorphic continuation to $ \mathfrak{D}$.
\end{prop}

\begin{rem} Note that for a domain $\mathfrak{D}$ as in Proposition \ref{prop6}, the meromorphic continuation of $\varphi(\psi(w))$ to $ \mathfrak{D}\cup\mathbb{T}_1\cup \mathfrak{D}^*$ is likewise the composition of two meromorphic functions, since by Proposition \ref{prop1}, $\varphi$ is meromorphic in $\overline{G}_1\cup\psi(\mathfrak{D}^*)$ and obviously $\psi( \mathfrak{D}\cup\mathbb{T}_1\cup \mathfrak{D}^*)\subset \overline{G}_1\cup\psi(\mathfrak{D}^*)$.
\end{rem}

Let $\rho_a\geq 0$ be the smallest number such that $\psi$ has an analytic continuation to
$\rho_a<|w|<\infty$, and let $\bar{\mu}\in [\rho_a, 1)$ be a number that has been fixed.  Suppose $z\in G_1$ and that the equation $\psi(w)=z$ has  no solutions in $\bar{\mu}<|w|<1$. In this case we assign $z\in C^{\bar{\mu}}_0$. Otherwise, the equation $\psi(w)=z$ has finitely many solutions of largest modulus, say  $v_{z,1},  \ldots,v_{z,s}$ ($s\geq 1$),  in $\bar{\mu}<|w|<1$ . Let $\beta_{z,k}$ be the multiplicity of $\psi$ at $v_{z,k}$.

For every integer $p\geq 1$, we define the subset $C^{\bar{\mu}}_p$ of $G_1$ by the relation
\begin{equation*}
z\in C^{\bar{\mu}}_p \Leftrightarrow \beta_{z,1}+\cdots+ \beta_{z,s}=p.
\end{equation*}

Finally, for $r\in [\rho_a,\infty)$, we define
\begin{equation*}
L_r:=\left\{z=\psi(w):|w|=r\right\}.
\end{equation*}
Note that for $r>\rho$, this definition of $L_r$ is equivalent to that given in (\ref{eq30}).

\begin{prop}\label{prop2} Suppose $\bar{\mu}$ is a number satisfying that $\rho_a\leq \bar{\mu}<\rho$ and having the property that $L_r\subset G_{1/r}$ for all $\bar{\mu}<r<1$. Then, $\mu\leq \bar{\mu}$, $\Sigma_p\supset C^{\bar{\mu}}_p$ for all $p\geq 1$, and
\[
\Phi(z)=v_{z,1},\quad z\in C^{\bar{\mu}}_1.
\]
Moreover, if $\mu=\bar{\mu}$, then $\Sigma_p=C^{\bar{\mu}}_p$ for all $p\geq 0$.

Suppose, in addition, that  there is a sequence $\{w_n\}_{n\geq 1}$, $\bar{\mu}<|w_n|<1$, such that
\begin{equation}\label{eq35}
\psi(w_{n+1})=\psi(1/\overline{w_n}),\quad n\geq 1.
\end{equation}
Then, $|w_n|>|w_{n+1}|$  and $\mu=\bar{\mu}=\lim_{n\to\infty}|w_n|$.
\end{prop}

\section{Examples}\label{sec3}

Two well-known sequences of polynomials orthogonal over the interior of an analytic Jordan curve are those corresponding  to $L_1$ a circle and $L_1$ an ellipse.  In both instances, the orthogonal polynomials  can be computed explicitly. The case of $L_1$ a circle is quite trivial and has been already discussed in Remark \ref{rem3} above. More interesting is the situation where $L_1$ is an ellipse, which without loss of generality can be assumed to have foci at $-1$ and $1$. Here (see, e.g., \cite[pp. 258-259]{Nehari}) $p_n$ is, up to a multiplicative constant, the $n$th Tchebichef polynomial of the second kind.

These examples are, however, of little relevance to us because in both of them $\Sigma_1=\Omega_\rho\cap G_1$, so that Theorem \ref{thm3}(a) reduces to the original result (\ref{eq8}) of Carleman.  We now provide two examples in which $\Sigma_1$ is actually larger than $\Omega_\rho\cap G_1$. In particular, we shall see that the inequalities   $\rho_a<\mu<\rho$ and $\mu<\rho_a$ are both possible.

\subsection{Cassini ovals}
Let $0<R<1$ be a number that has been fixed. The set $|z^2-1|=R$ consists of two disjoint analytic Jordan curves known as Cassini ovals.  One surrounds $1$, the other $-1$. Of these two, let us denote by $L_1$ the one encircling the point $1$.

Observe that the function
\[
\varphi(z):=(z^2-1)/R
\]
conformally maps $G_1$ (the interior of $L_1$) onto the unit disk. Given that $\varphi$ is an entire function, we have in view of Propositon \ref{prop1} that
\[
\mu=0.
\]

\begin{prop} \label{prop3} Let $a$ be the unique solution that the equation
\begin{equation*}
27x^4-18x^2-4\left(R+R^{-1}\right)x-1=0
\end{equation*}
has in the interval $(-1/3,0)$. Then
\begin{equation} \label{eq109}
h_\varphi(w)=\frac{(1-aw)w^2}{w-a},\quad w\in \overline{\mathbb{C}},
\end{equation}
\begin{equation}\label{Eq18}
\psi(w)=\sqrt{1+\frac{R(1-aw)w^2}{w-a}},\quad |w|>1,
\end{equation}
where the branch of the square root in (\ref{Eq18}) is chosen so that $\psi(1/a)=-1$.
\end{prop}

We shall see during the proof of this proposition that
\begin{equation}\label{eq42}
w-a+R(1-aw)w^2=-aR(w-b)^2(w-c),
\end{equation}
with
\begin{equation}\label{eq43}
b=\frac{\frac{3a^2+1}{2a}-\sqrt{\left(\frac{3a^2+1}{2a}\right)^2-4}}{2}, \quad c=\frac{1-3a^2}{2a}+\sqrt{\left(\frac{3a^2+1}{2a}\right)^2-4}\,,
\end{equation}
\begin{equation}\label{eq46}
-1<1/b<c<a<0,
\end{equation}
so that $\psi$ admits an analytic continuation to $\mathbb{C}\setminus [c,a]$ given by
\[
\psi(w)=\sqrt{-aR}\,(w-b)\sqrt{\frac{w-c}{w-a}}, \quad z\in \mathbb{C}\setminus [c,a].
\]
Moreover, $\rho=|b|^{-1}$, and so we have
\begin{equation}\label{eq120}
0=\mu<\rho_a=|c|<\rho=|b|^{-1}.
\end{equation}

It is not difficult to verify that the set
\[
\Gamma_R:=\left\{w\in\mathbb{D}_1:-R\leq h_\varphi(w)\leq 0\right\}
\]
is an analytic Jordan curve symmetric with respect to the real axis, intersecting it at $1/b$ and $0$. The function $h_\varphi(w)$ maps $\Gamma_R\cap\{z:\Im(z)>0\}$ onto $(-R,0)$ in an injective way. Hence for every $x\in (\sqrt{1-R^2},1)$, the equation
\begin{equation*}
h_\varphi(w)=\varphi(x)
\end{equation*}
has exactly two solutions in $\Gamma_R$, say $\omega_{x,1}$, $\omega_{x,2}$. These are distinct, and $\omega_{x,1}=\overline{\omega_{x,2}}$.

\begin{figure}
\centering
\includegraphics[scale=.42]{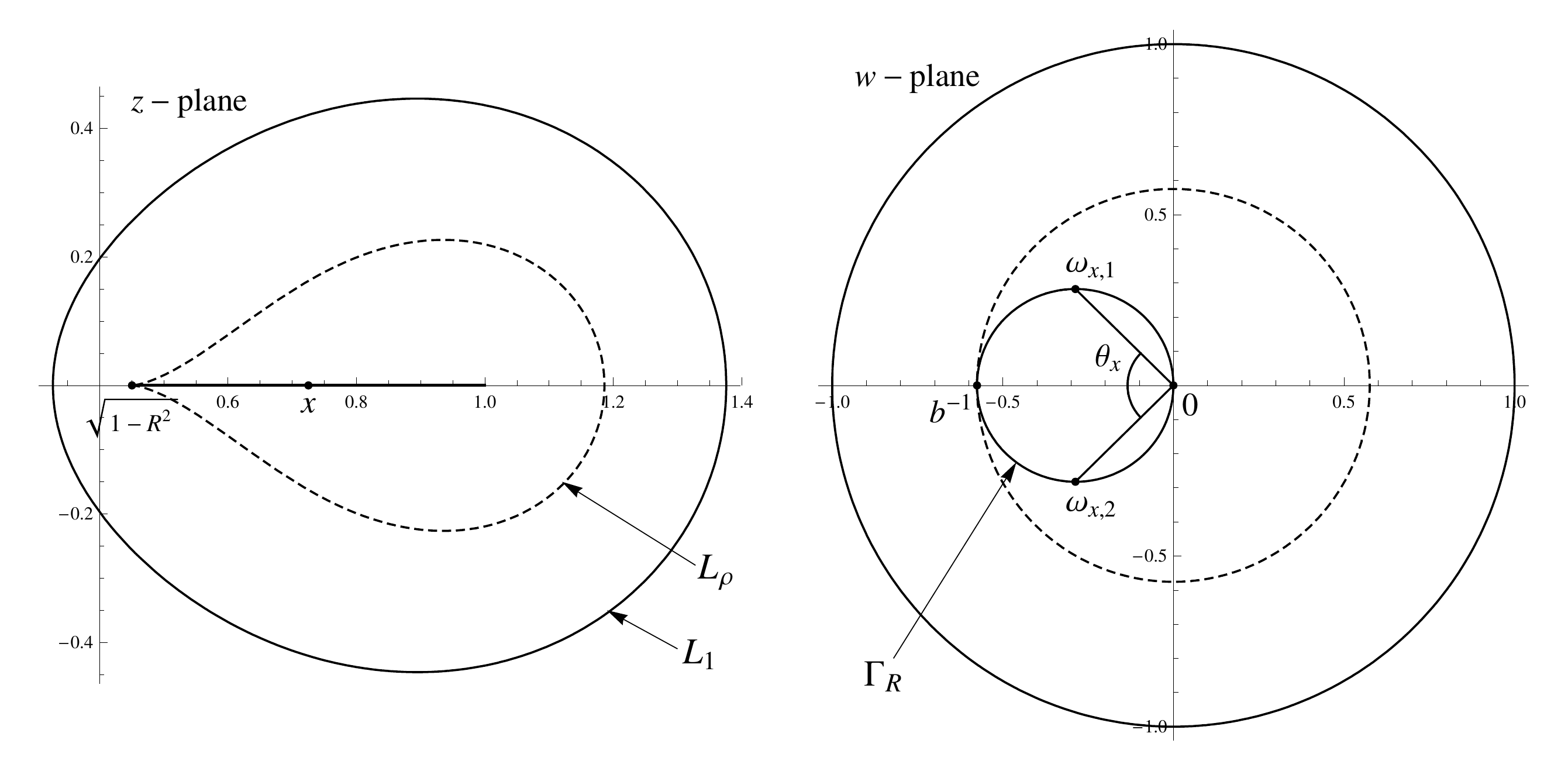}
\caption{Curve $\Gamma_R$ for $a=-.26$ ($R\approx 0.8926$, $\sqrt{1-R^2}\approx 0.4506$).}
\label{figura1}
\end{figure}
\begin{figure}
\centering
\includegraphics[scale=.7]{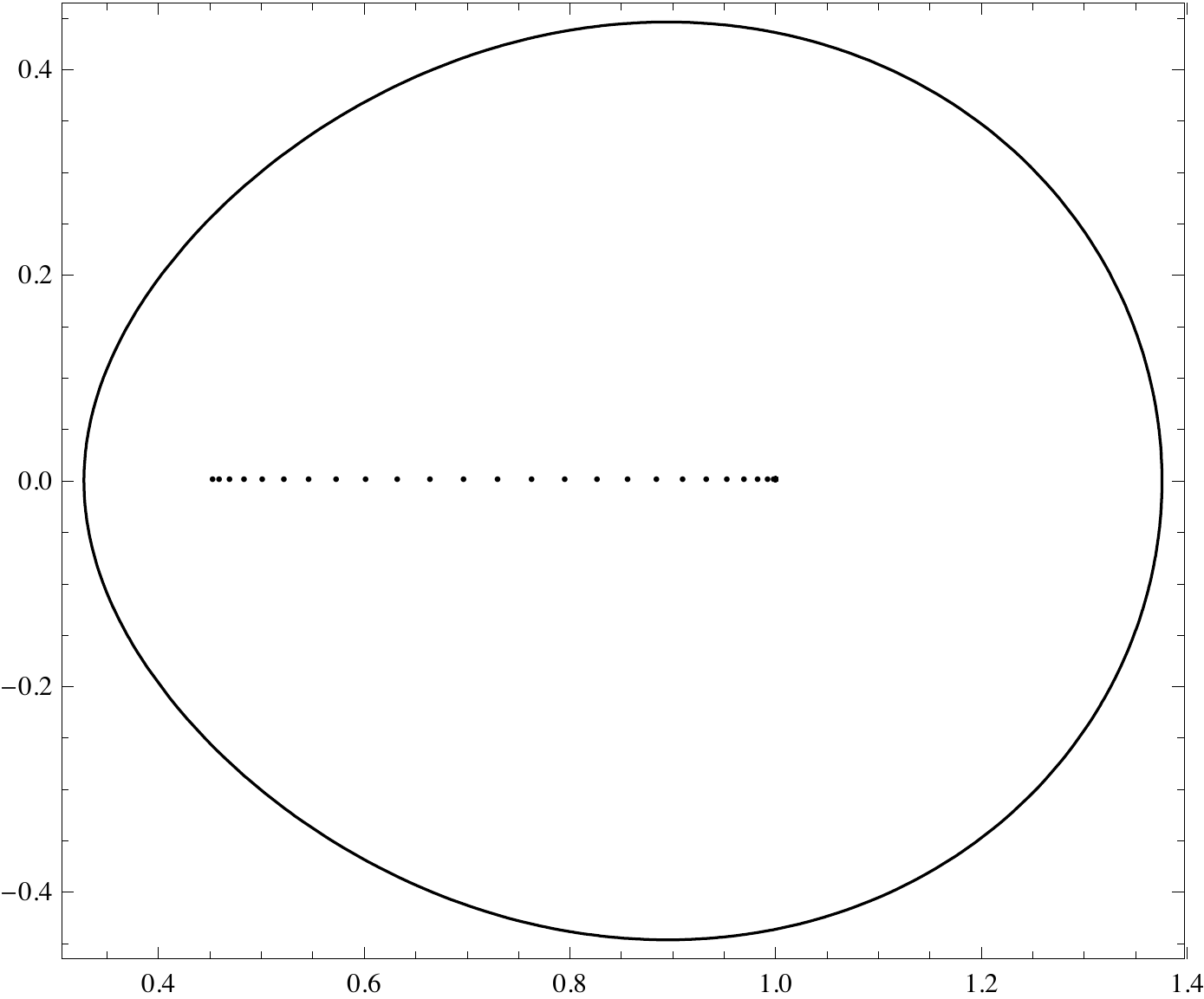}
\caption{Zeros of $p_{50}(z)$ for $a=-.26$ ($R\approx 0.8926$, $\sqrt{1-R^2}\approx 0.4506$).}
\label{figura2}
\end{figure}

Let $0<\theta_x<\pi$ be the angle formed by the two rays emanating from $0$ and passing through $\omega_{x,1}$, $\omega_{x,2}$ (see Figure \ref{figura1}). Recall that $\nu_n$ denotes the normalized counting measure of the zeros of $p_n$.

\begin{thm} \label{thm5}
\[
G_1=\Sigma_1\cup\Sigma_2,\quad \Sigma_2= [\sqrt{1-R^2},1],
\]
and $\{\nu_n\}_{n=1}^\infty$ converges in the weak*-topology to $\sigma+\delta_{1}/2$, where $\sigma$ is the measure supported on $\left[\sqrt{1-R^2},1\right]$ with distribution function
\begin{equation}\label{eq112}
\sigma([\sqrt{1-R^2},x])= \frac{\theta_x}{2\pi}, \quad \sqrt{1-R^2}\leq x\leq 1,
\end{equation}
and  $\delta_1$ is the unit point mass at $1$.
\end{thm}

Thus, in this example $\Sigma_0=\Sigma_p=\emptyset$ for all $p>2$.  The asymptotic formula of Theorem \ref{thm3}(a) holds with $\Phi(z)$ the algebraic function   analytic in $\mathbb{C}\setminus [\sqrt{1-R^2},1]$ that is solution of the  equation
\[
R(1-aw)w^2-(z^2-1)(w-a)=0,\quad
\Phi(-1)=1/a.
\]

In Figure \ref{figura2}, we have plotted the zeros of the polynomial $p_{50}$. They all seem to lie in the segment $[\sqrt{1-R^2},1]$, and only $26$ of them show up. This is corroborated by the following theorem, which we derive directly from the orthogonality property of the $p_n$'s.

\begin{thm} \label{thm6} For each $n\geq 0$,
\[
p_n(z)=(z-1)^{\lfloor n/2\rfloor}q_{n}(z),
\]
where $q_n(z)$ is a polynomial with $n-\lfloor n/2\rfloor$ simple roots, all lying in $(\sqrt{1-R^2},1)$.
\end{thm}

\subsection{Level curves of the inverse of a shifted Joukowsky transformation}

Let $R>2$ be fixed, and set
\begin{equation}\label{eq118}
L_1:=\left\{w-1+(w-1)^{-1}:|w|=R\right\}.
\end{equation}

From very well-known properties of the Joukowsky transformation $u \mapsto u+1/u$, it follows that $L_1$ is an analytic Jordan curve, with
\begin{equation}\label{eqq11}
\psi(w)=Rw-1+\frac{1}{Rw-1}, \quad z\in\overline{\mathbb{C}},
\end{equation}
mapping $\Delta_1$ conformally onto the exterior $\Omega_1$ of $L_1$. Moreover, $\psi$ maps both $\{w:|w-1/R|>1/R\}$ and $\{w:|w-1/R|<1/R\}$ conformally onto $\overline{\mathbb{C}}\setminus [-2,2]$, and for every $z\in \mathbb{C}$, the two solutions of the equation $z=\psi(w)$
are
\begin{equation}\label{eqq33}
v_{z,1}=\frac{z+2+\sqrt{z^2-4}}{2R}, \quad v_{z,2}=\frac{z+2-\sqrt{z^2-4}}{2R}.
\end{equation}

Note that $v_{z,1}$ and $\overline{v_{z,2}}$ are reflections of each other about the circle $|w-1/R|=1/R$, and that if we choose  the branch of the square root in (\ref{eqq33}) that is positive along $(2,\infty)$, then $|v_{z,1}|>|v_{z,2}|$ for every  $z\in\mathbb{C}\setminus [-2,2]$, with $v_{z,1}$ and $v_{z,2}$ lying, respectively, outside and inside the circle $|w-1/R|=1/R$.

\begin{figure}
\centering
\includegraphics[scale=.65]{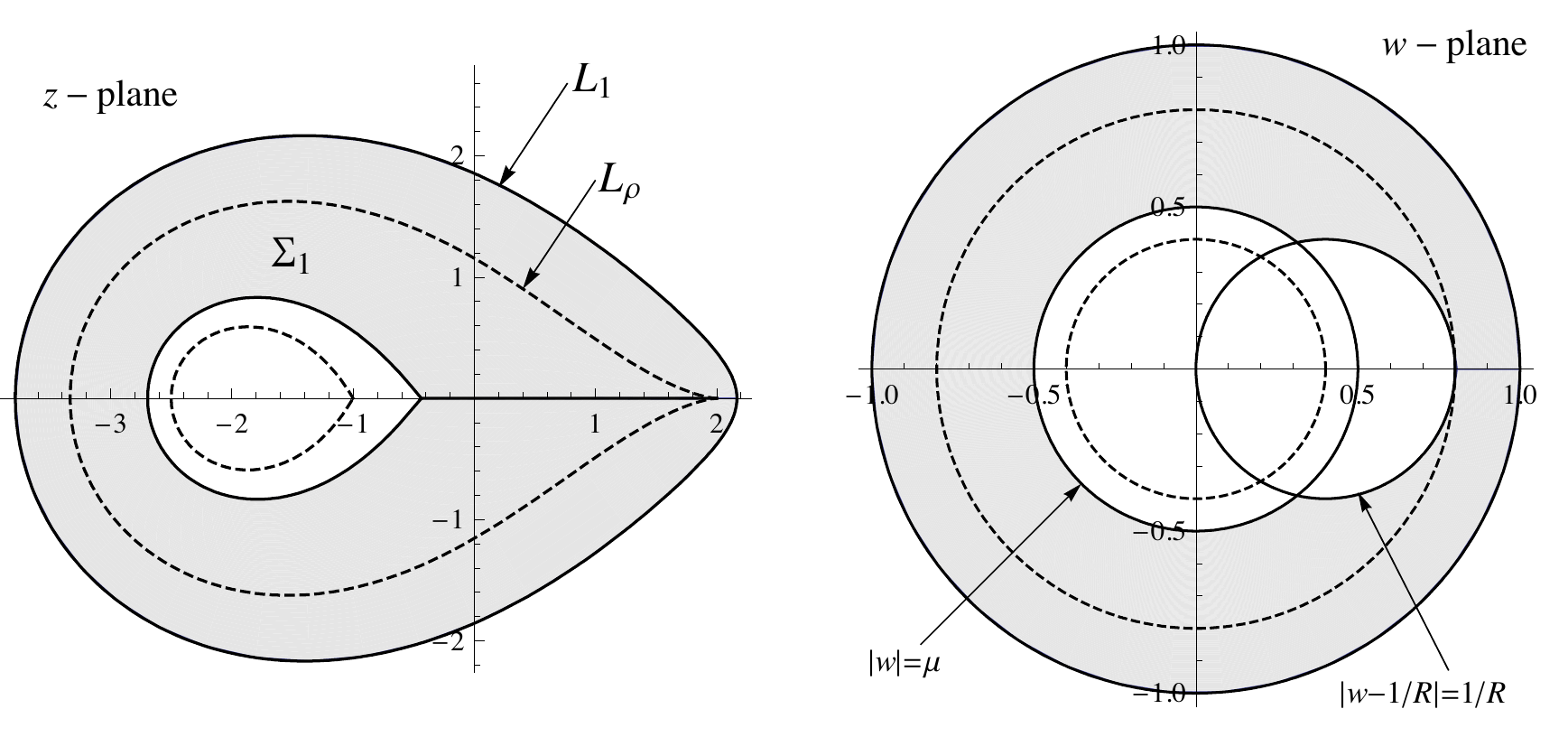}
\caption{Sets $\Sigma_1$, $\Sigma_2$ and $\Sigma_0$ for a curve $L_1$ defined as in (\ref{eq118}) for $R=2.5$.}\label{figura3}
\end{figure}
\begin{figure}
\centering
\includegraphics[scale=1]{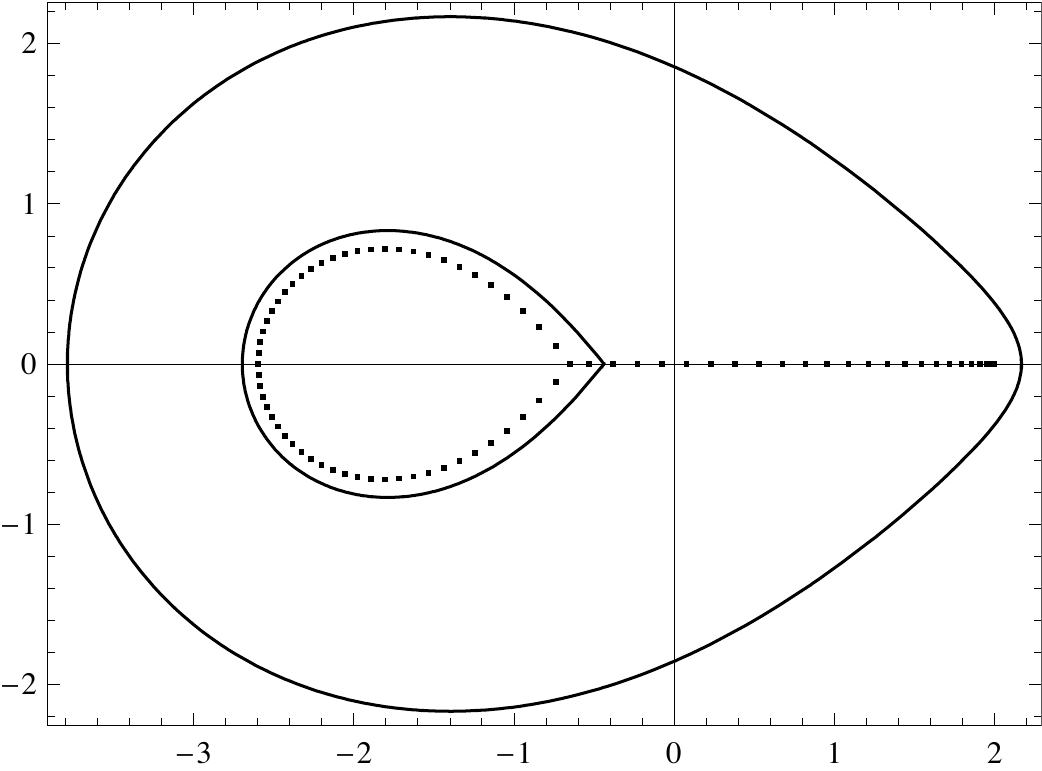}
\caption{Zeros of $p_{80}(z)$ for $L_1$ as in (\ref{eq118}) for $R=2.5$.}\label{figura4}
\end{figure}

We appeal to Proposition \ref{prop2} and find
\begin{thm}\label{thm8}
\begin{equation*}
\mu=\frac{R-\sqrt{R^2-4}}{2}
\end{equation*}
and  $G_1=\Sigma_0\cup\Sigma_1\cup\Sigma_2$, where $\Sigma_1$ is the image by $\psi$ of the set of points in $\mathbb{D}_1$ that lie exterior to both circles $|w|=\mu$ and $|w-1/R|=1/R$, and $\Sigma_2=(R^2\mu^2-2,2]$ (see Figure \ref{figura3}).
\end{thm}

Given that $\Sigma_1$ is connected, the function $\Phi$ is nothing but the analytic continuation of $\phi(z)=(z+2+\sqrt{z^2-4})/(2R)$, and so it follows from Carleman's formula (\ref{eq8}) and Theorem \ref{thm3}(a) that
\[
\lim_{n\to\infty}\frac{(2R)^{n+1}p_{n}(z)}{\sqrt{n+1}\left[z+2+\sqrt{z^2-4}\right]^n}
=\frac{z+\sqrt{z^2-4}}{\sqrt{z^2-4}}
\]
locally uniformly in $\Sigma_1\cup \overline{\Omega}_1$. Also, Theorem \ref{thm4} tells us that every point of $\partial \Sigma_1\cap G_1$ attracts zeros of the $p_n$'s, a fact illustrated in Figure \ref{figura4}.

Observe that, unlike the previous example in which we had $\mu<\rho_a$ (see (\ref{eq120})), we now have
\[
\frac{1}{R}=\rho_a<\mu<\rho=\frac{2}{R}.
\]

An interesting problem worth of consideration is that of determining the asymptotic behavior of $p_n$ in the interior of $\Sigma_0$. Naturally, one would attempt to derive such a behavior from (\ref{eq2}), but the situation is now considerably more difficult because for every $z\in \Sigma_0$, the function $(\varphi(\psi(w))-\varphi(z))^{-1}$ figuring under the integral sign in (\ref{eq2}) is a meromorphic function having  infinitely many poles in $|w|\leq \mu$ that accumulate over $\mu$. These poles, however, can be recursively found. These facts are summarized in the following result, whose proof illustrates the use of Proposition \ref{prop6} and which we think will serve the interested reader in his investigation of this problem.

Let $\varphi$ be the unique conformal map of $G_1$ onto the unit disk satisfying that $\varphi(-2)=0$, $\varphi'(-2)>0$.

\begin{prop} \label{prop8}
\begin{enumerate}[(a)]
\item The function $\varphi(\psi(w))$, originally defined in $\rho<|w|<1/\rho$, has a meromorphic continuation $h_\varphi(w)$ to $\overline{\mathbb{C}}\setminus\left\{\mu,1/\mu\right\}$, which satisfies (\ref{eq32}) for all $w\in\overline{\mathbb{C}}\setminus\left\{\mu,1/\mu\right\}$. The map $\varphi(z)$ has a meromorphic continuation to $\overline{\mathbb{C}}\setminus \left\{R^2-2\right\}$.

\item For every $z\in \Sigma_0$, the solutions that the equation $h_\varphi(w)=\varphi(z)$ has in $|w|\leq \mu$ are the odd-indexed elements of the two recursively generated sequences
\begin{equation}\label{eq122}
w_{n+1}=\frac{1}{R-\overline{w}_n},\quad n\geq 1,\quad w_1=\frac{z+2\pm\sqrt{z^2-4}}{2R}.
\end{equation}

\end{enumerate}
\end{prop}

Since $\psi(w)$ maps $\overline{\mathbb{C}}\setminus\left\{\mu,1/\mu\right\}$ onto $\overline{\mathbb{C}}\setminus \left\{R^2-2\right\}$, we get from Proposition \ref{prop8}(a) that  $h_\varphi(w)$ is the composition of the meromorphic continuation to $\overline{\mathbb{C}}\setminus \left\{R^2-2\right\}$ of $\varphi(z)$ with $\psi(w)$. The points $\mu$, $1/\mu$, and the point $R^2-2$ are non-isolated singularities of $h_\varphi(w)$ and $\varphi(z)$, respectively, since they are limits of poles of these functions.

\section{Proofs of the results of Section \ref{sec2}}\label{sec4}

We begin by reminding that if $\varphi_1$ and $\varphi$ are conformal maps of
$G_1$ onto $\mathbb{D}_1$, then they are related through a M\"{o}bius transformation, that is, for all $z\in G_1$,
\[
\varphi_1(z)=e^{i\theta}\frac{\varphi(z)-\varphi(z_0)}{1-\overline{\varphi(z_0)}\varphi(z)},\quad \varphi(z)=\frac{\varphi_1(z)+e^{i\theta}\varphi(z_0)}{e^{i\theta}+\overline{\varphi(z_0)}\varphi_1(z)},
\]
where $z_0\in G_1$  is such that $\varphi_1(z_0)=0$, and $\theta=\arg(\varphi_1'(z_0)/\varphi'(z_0))$.

Hence $\varphi(\psi(w))$ has a meromorphic continuation to the annulus $\mu<|w|<1$ if and only if so does $\varphi_1(\psi(w))$. Moreover, the meromorphic continuations $h_\varphi(w)$ and $h_{\varphi_1}(w)$ satisfy
\[
h_{\varphi_1}(w)-\varphi_1(z)=\frac{e^{i\theta}\left(1-|\varphi(z_0)|^2\right)}
{1-\overline{\varphi(z_0)}\varphi(z)}\cdot\frac{h_{\varphi}(w)-\varphi(z)}{1-\overline{\varphi(z_0)}h_{\varphi}(w)},
\]
so that for every $z\in G_1$, the equations (in the unknown $w$) $h_{\varphi_1}(w)=\varphi_1(z)$ and $h_{\varphi}(w)=\varphi(z)$ share the same solutions, multiplicities included. Therefore, the number $\mu$, the sets $\Sigma$, $\Sigma_p$, and the functions $\Phi(z)$, $r(z)$ as defined in Section \ref{sec2}, are independent of the choice of the interior conformal map $\varphi$.

Hereafter $\varphi$ is a conformal map of $G_1$ onto $\mathbb{D}_1$ that has been fixed. We shall employ the notation
\[
D_\epsilon(t):=\{z:|z-t|<\epsilon\}, \quad D^*_\epsilon(t):=D_\epsilon(t)\setminus\{t\}.
\]
As in Section  \ref{sec2} above, for $z\in \Sigma$, we denote by $\omega_{z,1},  \ldots, \omega_{z,s}$ the solutions to the equation
\begin{equation}\label{eq104}
h_\varphi(w)=\varphi(z)
\end{equation}
in the annulus $\mu<|w|<1$ that have largest modulus. The multiplicity of $h_\varphi(w)$ at $\omega_{z,k}$ is denoted by $\alpha_{z,k}$ and $\Sigma_p$, $p\geq 1$ is defined by relation (\ref{eq102}).

\begin{lem}\label{lem1}
Let $z\in \Sigma_p$, $p\geq 1$ and let $\mu'$ with  $\mu<\mu'<r(z)$ be a number satisfying that $h_\varphi(w)$ has no poles on $\mathbb{T}_{\mu'}$  and that the only solutions to (\ref{eq104}) that lie in $\mu'\leq |w|<1$ are precisely those of largest modulus  $\omega_{z,1},  \ldots, \omega_{z,s}$ ($1\leq s\leq p $). Let $\delta>0$ be so small that the closed disks $\overline{D_\delta(\omega_{z,k})}$, $1\leq k\leq s$, are pairwise disjoint and contained in the annulus $\mu'<|w|<1$, that $h_\varphi(w)$ has no poles on $\cup_{k=1}^s\overline{D_\delta(\omega_{z,k})}$ and that
\begin{equation}\label{eq26}
h'_\varphi(w)\not=0,\quad w\in D^*_\delta(\omega_{z,k}),\quad 1\leq k\leq s.
\end{equation}

There exists $\epsilon>0$ such that if $0<|\zeta-z|\leq \epsilon$, then the solutions to the equation
\begin{equation*}
h_\varphi(w)=\varphi(\zeta)
\end{equation*}
that lie in $\mu'\leq |w|\leq 1$ are simple and contained in $\cup_{k=1}^sD_\delta(\omega_{z,k})$, and each disk $D_\delta(\omega_{z,k})$ contains exactly $\alpha_{z,k}$ solutions.
\end{lem}

\begin{proof} Suppose $z$, $\mu'$ and  $\delta$ are as in the hypothesis of Lemma \ref{lem1}. Then, for
\begin{equation}\label{eq107}
K:=\{w:\mu'\leq |w|\leq 1\}\setminus \cup_{k=1}^sD_\delta(\omega_{z,k}),
\end{equation}
we have that  $m:=\min_{w\in K}|h_\varphi(w)-\varphi(z)|>0$. Select $\epsilon> 0$ such that $\overline{D_\epsilon (z)}\subset G_1$ and $|\varphi(\zeta)-\varphi(z)|<m$ for all $\zeta\in \overline{D_\epsilon (z)}$. Then, for this $\epsilon$, the conclusion of Lemma \ref{lem1} follows from Rouche's theorem and  (\ref{eq26}).
\end{proof}

\begin{cor}\label{cor1} Both $\Sigma$ and $\Sigma_1$ are open, the function $\Phi:z\mapsto \omega_{z,1}$ is analytic and univalent on $\Sigma_1$, and the function $r(z)$ is continuous on $G_1$.
\end{cor}

\begin{proof} That $\Sigma$ and $\Sigma_1$ are open is a clear consequence of Lemma \ref{lem1}, as well as the fact that $\Phi(z)$ is locally the inverse of $\varphi^{-1}\left(h_\varphi(w)\right)$. Therefore, $\Phi$ is analytic in $\Sigma_1$, and given that $\varphi(z)=h_\varphi(\Phi(z))$ for all $z\in\Sigma_1$, $\Phi(z)$ is one-to-one in $\Sigma_1$ and
\begin{equation}\label{eq108}
\Phi'(z)=\frac{\varphi'(z)}{h'_\varphi(\Phi(z))},\quad z\in\Sigma_1.
\end{equation}

The function $r(z)$ is by definition constant in $\Sigma_0$, and Lemma \ref{lem1} trivially yields that it is continuous in $\Sigma$. We prove now that it is also continuous at every point of $\partial \Sigma_0$. Suppose, on the contrary, that there exists $z\in \partial \Sigma_0$,  a sequence $\{z_n\}_{n=1}^\infty\subset \Sigma$ and a number $\mu_1>\mu$ such that $\lim_{n\to\infty}z_n=z$ and $r(z_n)\geq \mu_1>\mu$ for all $n\in\mathbb{N}$. For each $n$, let $\omega_n\in \mathbb{T}_{r(z_n)}$ be such that $h_\varphi(\omega_n)=\varphi(z_n)$. By extracting a subsequence if necessary, we may assume that $\lim_{n\to\infty}\omega_n=\omega$, with $\mu_1\leq |\omega|\leq 1$. But then, by the continuity of $\varphi$ and $h_\varphi$, we must have $h_\varphi(\omega)=\varphi(z)$. Given that $z\in G_1$, this is only possible if $|\omega|<1$, contradicting that $z\in \Sigma_0$.
\end{proof}

\begin{lem}\label{lem2}
\begin{enumerate}[(a)]
\item For every $z\in \Sigma$ and $\delta>0$, there exist $\epsilon>0$ and a constant $M$ such that \[
|p_n(\zeta)|\leq M\sqrt{n+1}\,[r(z)+\delta]^n,\quad \zeta\in D_\epsilon(z),\ n\geq 0.
\]

\item For every $z\in \Sigma_1$ there exist $\epsilon>0$ and $0<v<1$ such that
\begin{equation*}
p_{n}(\zeta)=\sqrt{n+1}\,\Phi'(\zeta)[\Phi(\zeta)]^n\left[1+O(v^n)\right]
\end{equation*}
uniformly in $\zeta\in D_\epsilon(z)$ as $n\to\infty$.

\item For every $\sigma\in(\mu,1)$ and $\delta>0$, there exists a constant $M_1$ such that for every $\zeta$ with  $r(\zeta)\leq \sigma$,
\begin{equation*}
 |p_n(\zeta)|\leq M_1\sqrt{n+1}(\sigma+\delta)^{n}, \quad  \ n\geq 0.
\end{equation*}
\end{enumerate}
\end{lem}

\begin{proof}
We first observe that
\[
\Sigma=\{z:\mu<r(z)<1\}= \cup_{k=1}^\infty\left\{z:\max\{\mu,\rho^k\}<r(z)<1\right\},
\]
and proceed to prove by mathematical induction on $k$ that if $k\geq 1$,
then Parts (a) and (b) of Lemma \ref{lem2} hold true for every  $z$ with $\max\{\mu,\rho^k\}<r(z)<1$, while
Part (c) holds true  for every $\sigma$ with $\max\{\mu,\rho^k\}<\sigma<1$. That this is true for $k=1$ clearly follows from (\ref{eq8}) and (\ref{eqq38}), since $\mu\leq \rho$ and
\[
\left\{z:\rho<r(z)<1\right\}=\Omega_\rho\cap G_1.
\]

Then, suppose it is true for some given  $k\geq 1$. Let $z\in \Sigma$ be a fixed number such that
\[
\max\{\mu,\rho^{k+1}\}<r(z)\leq \rho^k.
\]
Select $\eta >0$ so small that
\begin{equation}\label{eq9}
\rho(\rho+\eta)^{k}<r(z)\quad (\Rightarrow \rho+\eta<1).
\end{equation}

Let  $\omega_{z,1}, \ldots, \omega_{z,s}$ be the solutions to the equation $h_\varphi(w)=\varphi(z)$ in $\mu<|w|<1$ that have largest modulus, so that $|\omega_{z,k}|=r(z)$, $1\leq k\leq s$. Choose  $\mu'$ and $\delta$ satisfying the hypothesis of Lemma \ref{lem1}, with the particularity that
\begin{equation}\label{eq21}
\rho(\rho+\eta)^{k}<\mu', \quad \cup_{k=1}^sD_\delta(\omega_{z,k})\subset \{w:\mu'<|w|<(\rho+\eta/2)^k\}.
\end{equation}
Then, by the induction hypothesis that Lemma \ref{lem2}(c) holds whenever \[\max\{\mu,\rho^k\}<\sigma<1,\] there is a constant $M_1$ such that for every $\zeta$ with  $r(\zeta)\leq (\rho+\eta/2)^k$,
\begin{equation*}
|p_n(\zeta)|\leq M_1\sqrt{n+1}(\rho+\eta)^{kn},\quad n\geq 0,
\end{equation*}
and so we obtain from (\ref{eq2}) that 
\begin{equation}\label{eq106}
 p_{n}(\zeta)=\frac{\sqrt{n+1}\,\varphi'(\zeta)}{2\pi
 i}\oint_{\mathbb{T}_1}\frac{w^{n}dw}{h_\varphi(w)-\varphi(\zeta)}
 +O\left(\sqrt{n}(\rho+\eta)^{kn}\rho^n\right)
 \end{equation}
uniformly on $\{\zeta:r(\zeta)\leq (\rho+\eta/2)^k\}$ as $n\to\infty$.

 Now, corresponding to the numbers $z$, $\mu'$ and $\delta$ above, choose an $\epsilon>0$ for which the thesis of Lemma \ref{lem1} holds true, so that (recall (\ref{eq21}))
for all $\zeta\in \overline{D_\epsilon(z)}$,  $r(\zeta)<(\rho+\eta/2)^k$  and the function $(h_\varphi(w)-\varphi(\zeta))^{-1}$ is analytic on the compact set $K$ defined as in (\ref{eq107}). Hence we obtain from (\ref{eq106}) that  uniformly in $\zeta\in \overline{D_\epsilon(z)}$ as $n\to\infty$,
\begin{equation}\label{eq11}
\begin{split}
 p_{n}(\zeta)={}&\frac{\sqrt{n+1}\,\varphi'(\zeta)}{2\pi
 i}\oint_{\mathbb{T}_{\mu'}}\frac{w^{n}dw}{h_\varphi(w)-\varphi(\zeta)}+O\left(\sqrt{n}(\rho+\eta)^{kn}\rho^n\right)\\
{} & +\frac{\sqrt{n+1}\,\varphi'(\zeta)}{2\pi
 i}\sum_{k=1}^s\oint_{\partial D_{\delta}(\omega_{z,k})}\frac{w^{n}dw}{h_\varphi(w)-\varphi(\zeta)}\,.
\end{split}
\end{equation}

Now, the function $(h_\varphi(w)-\varphi(\zeta))^{-1}$ is continuous as a function of $(w,\zeta)$ on the compact set $K\times \overline{D_\epsilon(z)}$  and we obtain from (\ref{eq11}), (\ref{eq9}) and (\ref{eq21}) that uniformly in $\zeta\in \overline{D_\epsilon(z)}$ as $n\to\infty$,
\begin{align*}
p_{n}(\zeta)={}&O(\sqrt{n}\mu'^n) + O\left(\sqrt{n}[r(z)+\delta]^n\right) +O\left(\sqrt{n}(\rho+\eta)^{kn}\rho^n\right)\\
                     ={}&O\left(\sqrt{n+1}[r(z)+\delta]^n\right).
\end{align*}

If $z\in \Sigma_1$, i.e., if $s=1$, then every $\zeta\in D_\epsilon(z)$ belongs to $\Sigma_1$ as well, so that  $(h_\varphi(w)-\varphi(\zeta))^{-1}$ is analytic on $\overline{D_{\delta}(\omega_{z,1})}$, except at the point $\Phi(\zeta):=\omega_{\zeta,1}$, where it has a simple pole. Therefore (recall (\ref{eq108})),
\begin{align*}
\frac{1}{2\pi
 i}\oint_{\partial D_{\delta}(\omega_{z,1})}\frac{w^{n}dw}{h_\varphi(w)-\varphi(\zeta)}={}&(\omega_{\zeta,1})^n
\lim_{w\to \omega_{\zeta,1}}\frac{w-\omega_{\zeta,1}}{h_\varphi(w)-h_\varphi(\omega_{\zeta,1})}\\
={} &\frac{\Phi'(\zeta)[\Phi(\zeta)]^n}{\varphi'(\zeta)}.
\end{align*}
Hence we obtain from (\ref{eq11}) and (\ref{eq21}) that uniformly in $\zeta\in D_\epsilon(z)$ as $n\to\infty$,
\begin{equation*}
p_{n}(\zeta)=\sqrt{n+1}\,\Phi\,'(\zeta)[\Phi(\zeta)]^n\left[1+O\left(v^n\right)\right],
\end{equation*}
with $0<v=\mu'/(r(z)-\delta)<1$. Given that $\delta$ could have been chosen arbitrarily small, we have proven that Parts (a) and (b) of Lemma \ref{lem2} hold true if $\max\{\mu,\rho^{k+1}\}<r(z)<1$.

Now, suppose $\sigma$ is such that $\max\{\mu,\rho^{k+1}\}<\sigma\leq \rho^k$, and let $\delta>0$ be given. By the continuity of the function $r(z)$ and the fact that $r(z)$  approaches $1$ as $z$ approaches $\partial G_1$, we have that the set $\{z:r(z)=\sigma\}$ is compact, and we can find finitely many points $z_1,\ldots,z_m$ in this set, and positive numbers $\epsilon_1,\ldots,\epsilon_m$,  $M_1$, such that  $\{z:r(z)=\sigma\}\subset\, \cup_{j=1}^mD_{\epsilon_j}(z_j)$, and
\begin{equation}\label{eq22}
|p_n(\zeta)|\leq M_1\sqrt{n+1}\,(\sigma+\delta)^{n}
\end{equation}
for all $\zeta \in \cup_{j=1}^mD_{\epsilon_j}(z_j)$, $n\geq 0$. But the set $\{z:r(z)<\sigma\}$ is a bounded open set whose boundary is precisely $\{z:r(z)=\sigma\}$, so that  by the maximum modulus principle for analytic functions, (\ref{eq22}) also holds for all $\zeta$ with $r(\zeta)\leq \sigma$.
\end{proof}

\begin{proof}[Proof of Theorem \ref{thm3}] Part (a) of Theorem \ref{thm3} is equivalent to Lemma \ref{lem2}(b). We then pass to prove Part (b).
Let $z\in G_1$. From the definition of $r(z)$, we see that the function (in the variable $w$) $\left(h_{\varphi}(w)-\varphi(z)\right)^{-1}$ is analytic in the annulus $r(z)<|w|<1/\rho$, with a singularity on the circle $\mathbb{T}_{r(z)}$ in case $r(z)>0$, and therefore, it has a Laurent expansion in said annulus,  say $\sum_{k=-\infty}^{\infty}a_k(z)w^k$, whose coefficients
\begin{equation}\label{eq24}
a_{-n}(z)=\frac{1}{2\pi
 i}\oint_{\mathbb{T}_1}\frac{w^{n-1}dw}{h_\varphi(w)-\varphi(z)}, \quad n\geq 0
\end{equation}
satisfy
\begin{equation}\label{eq25}
\limsup_{n\to \infty}|a_{-n}(z)|^{1/n}=r(z).
\end{equation}

Let $\tau$ be a number satisfying that $r(z)<\tau<1$, and in case $r(z)\not=0$, that  $\tau\rho<r(z)$. By Lemma \ref{lem2}(c), we can find a constant $M$ such that
\begin{equation*}
|p_n(z)|\leq M\sqrt{n+1}\tau^{n},
\end{equation*}
which combined with (\ref{eq2}) and (\ref{eq24}) yields
\begin{equation*}
 p_{n}(z)=\sqrt{n+1}\,\varphi'(z)a_{-n-1}(z)
 +O\left(\sqrt{n}(\tau\rho)^n\right)\quad (n\to\infty).
\end{equation*}
This, in view of (\ref{eq25}) and the fact that $\tau$ can be taken arbitrarily closed to $r(z)$, forces $\limsup_{n\to\infty}|p_n(z)|^{1/n}=r(z)$ .
\end{proof}

\begin{proof}[Proof of Proposition \ref{prop1}] From (\ref{eq66}) and (\ref{eq70}) we get that
\begin{equation}\label{eq31}
\varphi(\psi(w))=\frac{1}{\overline{\varphi(\psi(1/\overline{w}))}},\quad \rho<|w|<1/\rho.
\end{equation}

If $\varphi(\psi(w))$ has a meromorphic continuation, denoted by $h_\varphi(w)$, to the domain $\mathfrak{D}$ (resp. to $\mathfrak{D}^*$), then, by virtue of (\ref{eq31}), the function $w\mapsto 1/\overline{h_\varphi(1/\overline{w})}$ provides the  meromorphic continuation of $\varphi(\psi(w))$ to $\mathfrak{D}^*$ (resp. to $\mathfrak{D}$),  and  (\ref{eq32}) is satisfied.

Suppose now that $\varphi(z)$ is meromorphic in $\overline{G}_1\cup\psi(\mathfrak{D}^*)$. Then the composition $\varphi(\psi(w))$, originally defined for $\rho<|w|<1/\rho$, now makes perfect sense for $z\in \mathfrak{D}^*$, and it is obviously meromorphic.  Reciprocally, if $\varphi(\psi(w))$ has a meromorphic continuation $h_\varphi(w)$ to $\mathfrak{D}^*$, then $h_\varphi(\phi(z))$ is a meromorphic function in $\psi(\mathfrak{D}^*)$, and
\[
h_\varphi(\phi(z))=\varphi(\psi(\phi(z)))=\varphi(z),\quad z\in \Omega_1\cap G_{1/\rho}.
\]
\end{proof}

\begin{proof}[Proof of Proposition \ref{prop6}] By the first inclusion in Eq. (\ref{eqq31}), the composition $\varphi(\psi(w))$ is well-defined and analytic in $\mathfrak{D}_a$. Hence there exists a largest number  $t_0\in[a,b]$ such that $\varphi(\psi(w))$ has a meromorphic continuation to every $\mathfrak{D}_{t}$ with $a\leq t<t_0$. Suppose $t_0<b$. From assumption (\ref{eqq31}), we see that $\psi(\overline{\mathfrak{D}}_{t_0})\subset \overline{G}_1\cup \psi(\mathfrak{D}^*_{t_0})$, which combined with assumption (\ref{eqq32}) yields the existence of some $t_0<t_1<b$ such that $\psi(\mathfrak{D}_{t_1})\subset \overline{G}_1\cup \psi(\mathfrak{D}^*_{t_0})$. Since $\varphi(\psi(w))$ is meromorphic in $\mathfrak{D}_{t_0}$, by Proposition \ref{prop1}, $\varphi(z)$ is then meromorphic in $\overline{G}_1\cup \psi(\mathfrak{D}^*_{t_0})$, so that the composition $\varphi(\psi(w))$ is well-defined and meromorphic in $\mathfrak{D}_{t_1}$, contradicting the definition of $t_0$. Hence $t_0=b$ and $\varphi(\psi(w))$ has a meromorphic continuation to $\mathfrak{D}=\cup_{a\leq t<b}\mathfrak{D}_{t}$.
 \end{proof}

\begin{proof}[Proof of Proposition \ref{prop2}]  That $\mu\leq \bar{\mu}$  follows by applying Proposition \ref{prop6} to the collection of annuli $\mathfrak{D}_t:=\{w:\rho+\bar{\mu}-t<|w|<1\}$, $\bar{\mu}\leq t<\rho$.

To prove that $\Sigma_p\supset C_p^{\bar{\mu}}$, $p\geq 1$, we first make a couple of observations. The first one is that, in view of Proposition \ref{prop1}, $\varphi$ admits a meromorphic continuation (also denoted by $\varphi$) to $G_{1/\mu}$,  and since $L_r\subset G_{1/r}$ for $\bar{\mu}<r<1$,  we then have $\psi(\bar{\mu}<|w|<1/\mu)\subset G_{1/\mu}$ and
\begin{equation}\label{eqq3}
h_\varphi(w)=\varphi(\psi(w)),\quad \bar{\mu}<|w|<1/\mu.
\end{equation}
The second observation is stated as a claim.

\begin{claim} 
If $z\in G_1$ and $w_z$ with $\bar{\mu}<|w_z|<1$  are such that $h_\varphi(w_z)=\varphi(z)$, then either  $\psi(w_z)=z$ or the equation $h_\varphi(w)=\varphi(z)$ has a solution in $|w_z|<|w|<1$.
\end{claim}

In effect, suppose first that $\psi(w_z)\in G_1$. Then, by (\ref{eqq3}) and the fact that $h_\varphi(w_z)=\varphi(z)$, we must have $\psi(w_z)=z$. Next, assume  $\psi(w_z)\not\in G_1$.  Then  $\psi(w_z)\not\in L_1 =\partial G_1$ either, because $\varphi$ maps $L_1$ onto the unit circle and $|h_\varphi(w_z)|=|\varphi(\psi(w_z))|=|\varphi(z)|<1$. Moreover, since $\psi(w_z)\in L_{|w_z|}\subset G_{1/|w_z|}$ and $\psi$ maps $1< |w|<1/|w_z|$ conformally onto $G_{1/|w_z|}\setminus \overline{G}_1$, we see that there is a unique number $w'_z$ with $|w_z|<|w'_z|<1$ such that $\psi\left(1/\overline{w'_z}\right)=\psi(w_z)$.  By (\ref{eqq3}) and (\ref{eq32}), we then have
\begin{equation*}
h_\varphi(w'_z)=\frac{1}{\overline{h_\varphi\left(1/\overline{w'_z}\right)}}
=\frac{1}{\overline{h_\varphi(w_z)}}=\frac{1}{\overline{\varphi\left(z\right)}}\,.
\end{equation*}

This implies that $\psi(w'_z)\not\in\overline{G}_1$, which combined with the fact that $\psi(w'_z)\in L_{|w'_z|}\subset G_{1/|w'_z|}$ yields the existence of a unique $w''_z$ with $|w'_z|<|w''_z|<1$ such that $\psi\left(1/\overline{w''_z}\right)=\psi(w'_z)$, and so\begin{equation*}
h_\varphi(w''_z)=\frac{1}{\overline{h_\varphi\left(1/\overline{w''_z}\right)}}
=\frac{1}{\overline{h_\varphi(w'_z)}}=\varphi(z)\,,
\end{equation*}
which proves the claim.

We now proceed to prove that
\begin{equation}\label{eqq5}
C^{\bar{\mu}}_p=\left\{z\in \Sigma_p: r(z)>\bar{\mu}\right\}.
\end{equation}

Suppose  $z\in \Sigma_p$ is such that $r(z)>\bar{\mu}$, that is, $z\in G_1$ and there are finitely many numbers $\omega_{z,1}, \ldots,\omega_{z,s} $, with $\bar{\mu}<r(z)=|\omega_{z,1}|=\cdots=|\omega_{z,s}|<1$, which are the only solutions that the equation $h_\varphi(w)=\varphi(z)$ has in  $|\omega_{z,1}|\leq|w|<1$, and moreover $\sum_{k=1}^s\alpha_{z,k}=p$, with $\alpha_{z,k}$ being the multiplicity of $h_\varphi$ at $\omega_{z,k}$.

Then, by $(\ref{eqq3})$, the only possible solutions that the equation $\psi(w)=z$ could have in  $|\omega_{z,1}|\leq|w|<1$ are precisely these $\omega_{z,k}$. As a matter of fact, in view of the claim proven above, we have $\psi(\omega_{z,k})=z$ for all $1\leq k\leq s$, and it clearly follows from (\ref{eqq3}) that the multiplicity of $\psi$ at $\omega_{z,k}$ is $\alpha_{z,k}$.  Thus, $z\in C^{\bar{\mu}}_p$.

Assume now that $z\in C^{\bar{\mu}}_p$, that is, $z\in G_1$ and there are finitely many numbers $v_{z,1}, \ldots,v_{z,s} $, with $\bar{\mu}<|v_{z,1}|=\cdots=|v_{z,s}|<1$, which are the only solutions that the equation $\psi(w)=z$ has in $|v_{z,1}|\leq|w|<1$, and moreover $\sum_{k=1}^s\beta_{z,k}=p$, with $\beta_{z,k}$ being the multiplicity of $\psi$ at $v_{z,k}$. These $v_{z,k}$'s are the only possible solutions that the equation $h_\varphi(w)=\varphi(z)$ could have in $|v_{z,1}|\leq |w|<1$, because by the claim proven above, among such solutions those of largest modulus must be mapped by $\psi$ to $z$. Moreover, by (\ref{eqq3}), we have that for all $1\leq k\leq s$, $h_\varphi(v_{z,k})=\varphi(z)$ and that $\beta_{z,k}$ is the multiplicity of $h_\varphi$ at $v_{z,k}$.  Hence $z\in \Sigma_p$, and (\ref{eqq5}) is proven.

Since every element of $\Sigma_p$, $p\geq 1$ satisfies $r(z)>\mu$, (\ref{eqq5}) implies that if $\mu=\bar{\mu}$, then $\Sigma_p=C^{\bar{\mu}}_p$ for all $p\geq1$, which in turn implies that $\Sigma_0=C^{\bar{\mu}}_0$.

Finally, suppose that a sequence $\{w_n\}_{n\geq 1}$ satisfying (\ref{eq35}) is found. Then,  given that $L_r\subset G_{1/r}$ for  $\bar{\mu}<r<1$, we have $\psi\left(1/\overline{w_n}\right)=\psi(w_{n+1})\in G_{1/|w_{n+1}|}$, so that $|w_n|>|w_{n+1}|$, $n\geq 1$. Moreover, in view of (\ref{eqq3}) and (\ref{eq32}),
\begin{equation*}	 h_\varphi(w_{n+2})=h_\varphi\left(1/\overline{w_{n+1}}\right)=\frac{1}{\overline{h_\varphi(w_{n+1})}}=\frac{1}{\overline{h_\varphi\left(1/\overline{w_{n}}\right)}}=h_\varphi(w_n),\quad n\geq 1.
\end{equation*}
Hence $h_\varphi(w)$ remains constant along an infinite set of points contained in $\bar{\mu}<|w|<1$. This is only possible if $\lim_{n\to\infty}|w_n|=\bar{\mu}=\mu$.
\end{proof}

\begin{proof}[Outline of the proof of Theorems \ref{thm4} and \ref{thm2}] In order to derive Theorems  \ref{thm4} and  \ref{thm2} from Theorem \ref{thm3}, one needs first to establish several structural properties of the sets  $\Sigma_p$, $p\geq 1$. These properties have been previously established for different, but similarly  defined sets. For instance,  in \cite{Ullman2}, Ullman studied the zero distribution of the Faber polynomials $F_n(z)$, $n=0,1,2,\ldots$, associated with a Laurent series about $\infty$ of the form
\begin{equation}\label{eq101}
g(w)=w+b_0+b_1w^{-1}+b_2 w^{-2}+\cdots
\end{equation}
with radius of convergence $\varrho:=\limsup_{n\to\infty}|b_n|^{1/n}<\infty$. The function $g$ is locally invertible at $\infty$, and if $g^{-1}(z)$ denotes its inverse, then $F_n(z)$ is defined as the polynomial part of the Laurent expansion at $\infty$ of $[g^{-1}(z)]^n$.

For each $p\geq 1$, Ullman introduced the set $C_p$ consisting of those points $z\in \mathbb{C}$ for which the solutions of largest modulus that the equation $g(w)=z$ has in $|w|>\varrho$, say $u_{z,1},\ldots,u_{z,s}$,  have total multiplicity $p$. Note the similarity of this definition with that of $\Sigma_p$ given in (\ref{eq102}). Setting $C_0:=\mathbb{C}\setminus \cup_{p\geq 1}C_p$, $\Psi(z):=u_{z,1}$ for all $z\in C_1$ and
\begin{equation*}
\tilde{\varrho}(z):=\left\{\begin{array}{ll}
               |u_{z,1}|, & z\in \cup_{p\geq 1}C_p, \\
               \varrho,  & z\in C_0,
             \end{array}\right.
\end{equation*}
Ullman proved that (see (3.7), (3.8), (5.1) and (5.4) in \cite{Ullman2})
\begin{equation}\label{eq126}
\limsup_{n\to\infty}|F_n(z)|^{1/n}=\tilde{\varrho}(z),\quad z\in\mathbb{C},
\end{equation}
and more specifically, for points in $C_1$, that
\begin{equation}\label{eq127}
\lim_{n\to\infty}F_n(z)/[\Psi(z)]^n=1
\end{equation}
locally uniformly on $C_1$. These asymptotic formulas are the analogue of Theorem \ref{thm3} for the Faber polynomials.

Ullman also proved that the sets $C_p$ have the following properties \cite[Lemmas 4.1, 4.2]{Ullman2}: \emph{Every  $z\in C_p$, $p\geq 1$ has a neighborhood that is fully contained in $\cup_{q=1}^p C_ q$.  Every $C_p$ with $p>1$ has empty interior. Every neighborhood of a point $z\in C_p$, $p>1$ contains points that are not in $C_1$.}

Combining these properties with (\ref{eq126})-(\ref{eq127}),  Ullman succeeded in proving that \cite[Theorem 1(b)]{Ullman2} every point of $\partial C_1$ is an accumulation point of the zeros of the $F_n$'s. Following Ullman's arguments, one can easily see that the properties just stated for the $C_p$'s are word for word valid for the sets $\Sigma_p$ as well, and that these properties in conjunction with Theorem \ref{thm3} imply the validity of Theorem \ref{thm4}.

In another paper \cite{Ullman1} dealing with the limiting behavior of the eigenvalues of Toeplitz matrices associated with a semi-infinite Laurent series of the form  $\sum_{n=-\infty}^k c_nw^n$ ($\limsup_{n\to\infty}|c_{-n}|^{1/n}<\infty$), Ullman  considered the smallest possible $\tau\geq 0$ for which there exists a meromorphic function $F(w)$  on $|w|>\tau$ having this expansion at $\infty$. He defined a corresponding set $C$  that for the case $k=1$ (i.e., a simple pole at $\infty$) consists of those points $z\in \mathbb{C}$ for which the equation $F(w)=z$ has exactly one solution in $|w|>\tau$ of largest modulus, and this solution is simple (see the definition of the set $C$ in \cite[Definition 1]{Ullman1}).  He proved two important lemmas \cite[Lemmas 7 and 8]{Ullman1} about the structure of the boundary of the set $C$, which can  be established in a similar way for both the set $C_1$ (i.e., the $C_p$ corresponding to  $p=1$) and the set $\Sigma_1$.
Using the extension of these lemmas to $C_1$ (see \cite[Lemmas 2.2 and 2.4]{KS}) together with Ullman's asymptotic formulas (\ref{eq126})-(\ref{eq127}), Kuijlaars and Saff proved the analogue of Theorem \ref{thm2} for the Faber polynomials $F_n$ associated with (\ref{eq101}) (see \cite[Theorems 1.3, 1.4 and 4.1]{KS}). Their arguments are based on general facts of logarithmic potential theory and can be used essentially without variation to derive our Theorem \ref{thm2}.
\end{proof}

\section{Proofs of the results of Section \ref{sec3}}\label{sec5}

\begin{proof}[Proof of Proposition \ref{prop3}] Given that $\varphi(z)=(z^2-1)/R$ is an entire function, Proposition \ref{prop1} implies that  $\mu=0$ and $h_\varphi(w)$ is meromorphic in $\mathbb{C}\setminus\{0\}$. By uniqueness of the meromorphic continuation, we then have
\begin{equation}\label{eq33}
h_\varphi(w)=\varphi(\psi(w))=\frac{[\psi(w)]^2-1}{R},\quad \rho<|w|<\infty.
\end{equation}

This and (\ref{eq32}) imply that $h_\varphi(w)$ is indeed a meromorphic function in $\overline{\mathbb{C}}$, whose only poles are $\infty$ and some point $a$, $0<|a|<1$, and whose only zeros are $1/\overline{a}$ and  $0$. $\infty$ and $0$ are of multiplicity $2$, while $a$ and $1/\overline{a}$ are simple. Hence for some complex number $\beta$,
\begin{equation*}
h_\varphi(w)=\frac{\beta(1-\overline{a}w)w^2}{w-a},\quad w\in \overline{\mathbb{C}}.
\end{equation*}

By symmetry, $\psi(w)=\overline{\psi(\overline{w})}$, which in view of the normalization $\psi'(\infty)>0$ implies that $\psi$ maps $(-\infty,-1)$ onto $(-\infty,\sqrt{1-R})$. Hence $-1<a<0$. Also, given that $|h_\varphi(w)|=1$ for $|w|=1$ and that by (\ref{eq33}) 
\[
\lim_{w\to\infty}h_\varphi(w)/w^2>0,
\]
we then must have $\beta=1$, and so (\ref{eq109}) is proven.

Equality (\ref{Eq18}) follows directly from  (\ref{eq109}) and (\ref{eq33}). To find the value of $a$, first observe that $0$ lies outside the curve $L_1$. Let $b$ be the point in $|w|>1$ such that $\psi(b)=0$ (then, $1/a<b<-1$). By (\ref{eq33}), $b$ is a double zero of $1+Rh_\varphi(w)$, so that (\ref{eq42}) holds for some $c$, and the relations
\begin{align}
aRb^3-Rb^2-b+a & =0\nonumber\\
3aRb^2-2Rb-1 &=0\nonumber\\
2ab^2-\left(3a^2+1\right)b+2a & =0 \label{eq38}
\end{align}
are satisfied. From these we get
\begin{equation*}
27a^4-18a^2-4\left(R+R^{-1}\right)a-1=0,
\end{equation*}
and it is easy to see that this Eq. (in the unknown $a$) has only two real solutions, one positive, the other contained in $(-1/3,0)$. This completes the proof of Proposition \ref{prop3}.

The equalities in (\ref{eq43}) follow from (\ref{eq38}) and (\ref{eq42}). Also from (\ref{eq42}) and Vieta's formulas we obtain the relations
\[
1/a=1/c+2/b,\quad 2b-2/b=1/c-c,
\]
which, given that $b<-1$, forces the inequalities in (\ref{eq46}) to hold true.
\end{proof}

\begin{proof}[Proof of Theorem \ref{thm5}]
We first observe that if $|\xi|<1$, then the equation
\begin{equation}\label{eq47}
\xi=\frac{(1-aw)w^2}{w-a}
\end{equation}
has exactly two roots (counting multiplicities) in $|w|<1$. To see this, suppose $w_1$, $w_2$ and $w_3$ are the roots of (\ref{eq47}). Then, not all can be contained in $|w|<1$, for in such a case $|(1-aw_j)/(w_j-a)|>1$, $1\leq j\leq 3$, which together with (\ref{eq47}) yields
\[
|\xi|^{1/2}>|w_j|,\quad 1\leq j\leq 3.
\]
Since $w_1w_2w_3=\xi$, we would have $|\xi|<|\xi|^{3/2}$, contradicting the assumption that $|\xi|<1$.

Assume now that $|w_1|>1$, $|w_2|>1$. Denoting by $\varphi^{-1}:\mathbb{D}_1\to G_1$ the inverse of $\varphi(z)=(z^2-1)/R$,  we get from (\ref{eq33}) that
\[
\psi(w_1)=\psi(w_2)=-\varphi^{-1}(\xi),
\]
so that $w_1=w_2$. Since $h'_\varphi(w)$ only vanishes at $b$, $1/b$ and $0$, we must have $w_1=w_2=b$, so that by (\ref{eq33}), $\xi=-1/R$, contradicting that $|\xi|<1$.

Thus, being $\varphi$ a bijection from $G_1$ to $\mathbb{D}_1$, we conclude that $G_1=\Sigma_1\cup\Sigma_2$.
For $|\xi|<1$, let $w_{\xi,1}$ and $w_{\xi,2}$ denote the two solutions of  (\ref{eq47}) lying in $|w|<1$. To prove that $\Sigma_2=[\sqrt{1-R^2},1]$, we prove the equivalent statement that
\begin{equation}\label{eq111}
S:=\left\{|\xi|<1:|w_{\xi,1}|=|w_{\xi,2}|\right\}=[-R,0].
\end{equation}

Suppose $\xi\in S$. From (\ref{eq47}) we obtain that for $j=1,2$,
\[
\Re(w_{\xi,j})=
\frac{|w_{\xi,j}|^4 + a^2 |w_{\xi,j}|^6 - a^2 |\xi|^2 - |w_{\xi,j}|^2 |\xi|^2}{2a \left(|w_{\xi,j}|^4 - |\xi|^2\right)}.
\]
Hence $w_{\xi,1}=\overline{w_{\xi,2}}$, and since $h_\varphi(\overline{w})=\overline{h_\varphi(w)}$, we deduce that $\xi$ must be real, and consequently, the point $\xi\in(-1,1)$ belongs to $S$ if and only if Eq. (\ref{eq47}) has either a double real root in $(-1,1)$, or no real roots in $(-1,1)$.

Since $h'_\varphi(w)$ only vanishes at $b$, $1/b$ and $0$, it follows that Eq. (\ref{eq47}) has a double root in $(-1,1)$ only for $\xi=-R=h_\varphi(1/b)$, $\xi=0=h_\varphi(0)$. On the other hand, considering $h_\varphi(x)$ as a function of the real variable $x$, and analyzing the sign changes of $h'_\varphi(x)$ in $(-1,1)$, it is easy to see that Eq. (\ref{eq47}) has no real roots in $(-1,1)$ if and only if $\xi\in(-R,0)$. Thus, (\ref{eq111}) is proven.

Since $\Sigma_0=\emptyset$, Theorem \ref{thm2} guarantees the convergence of  $\{\nu_n\}_{n=1}^\infty$ in the weak*-topology to a measure $\lambda$ supported on $[\sqrt{1-R^2},1]$ and having logarithmic potential
\begin{equation}\label{eq113}
U^{\lambda}(z) = \Re\left(\log\left[\phi'(\infty)/\Phi(z)\right]\right),\quad z\in G_1\setminus [\sqrt{1-R^2},\infty),
\end{equation}
with the convention $0<\arg(\phi'(\infty)/\Phi(z))<2\pi$.

We proceed to prove that $\lambda=\sigma+\delta_1/2$, with $\sigma$ as in (\ref{eq112}), for which we use a well-known formula \cite[Theorem II.1.4]{Saff} that allow a measure to be recovered from its logarithmic potential.

By the continuity of the pair of complex conjugate solutions that  Eq. (\ref{eq47}) has as the parameter $\xi$ varies in the closed interval $[-R,0]$, it is clear that
\[
\Gamma_R:=\left\{w\in\mathbb{D}_1:-R\leq h_\varphi(w)\leq 0\right\}
\]
is a Jordan curve symmetric with respect to the real axis, intersecting it at $1/b$ and $0$. It is easy to see that $\Gamma_R$ is in fact an analytic curve.

The function $\Phi(z)$ maps $G_1\setminus [\sqrt{1-R^2},1]$ conformally onto the portion of the unit disk that lies exterior to $\Gamma_R$. Moreover, if for $x\in (\sqrt{1-R^2},1)$,   $\omega_{x,1},\omega_{x,2}\in\Gamma_R$ are the two complex conjugate solutions that the equation $h_\varphi(w)=\varphi(x)$ has in $\mathbb{D}_1$ (say, with $\Im{\omega_{x,1}}>0$), then
\[
\lim_{t\to 0+}\Phi(x + it)=\omega_{x,1}, \quad \lim_{t\to 0-}\Phi(x+ it)=\omega_{x,2},
\]
and we obtain from these two equalities, (\ref{eq113}) and Theorem II.1.4 of \cite{Saff} that for all $ \sqrt{1-R^2}<x<1$,
\[
\lambda([\sqrt{1-R^2},x])= \frac{\theta_x}{2\pi},
\]
where  $0<\theta_x<\pi$ is the angle formed by the two rays emanating from $0$ and passing through $\omega_{x,1}$, $\omega_{x,2}$. Given that $\lim_{x\to1-}\theta_x=\pi$, we must have $\lambda(\{1\})=1/2$, completing the proof of Theorem \ref{thm5}.
\end{proof}

\begin{proof}[Proof of Theorem \ref{thm6}] Because $G_1$ is symmetric about the real axis, each $p_n$ has real coefficients. Let $n\geq 0$ be an integer. Combining the orthogonality property of $p_n$ with Green's formula (see, e.g., \cite[p. 241]{Nehari}) and using that  $(\overline{z}^2-1)(z^2-1)=R^2$ for  $z\in L_1$, we obtain that for $1\leq m\leq \lfloor n/2\rfloor$,
\begin{align*}
0 & =\int_{G_1}p_n(z)\overline{z^{2m-1}}dxdy
=\frac{1}{4im}\int_{L_1}p_n(z)\overline{z^{2m}}dz\\
 &=\frac{1}{2i(m+1)}\int_{L_1}p_n(z)\frac{(R^2+z^2-1)^m}{(z-1)^m(z+1)^m}dz.
 \end{align*}
Hence by the Cauchy integral formula, $p^{(j)}_n(1)=0$, $0\leq j\leq \lfloor n/2\rfloor-1$. Therefore,
$p_n(z)=(z-1)^{\lfloor n/2\rfloor}q_{n}(z)$, where $q_{n}(z)$ is a polynomial of degree $n-\lfloor n/2\rfloor$.

Similarly, we obtain that for $0\leq m\leq n-\lfloor n/2\rfloor-1$,
\begin{align*}
0 & =\int_{G_1}p_n(z)\overline{z^{2m}}dxdy=\frac{1}{2i(2m+1)}\int_{L_1}p_n(z)\overline{z^{2m+1}}dz\\
 &=\frac{1}{2i(m+1)}\int_{L_1}p_n(z)\left[\frac{R^2}{z^2-1}+1\right]^m\sqrt{\frac{z+\sqrt{1-R^2}}{1+z}}\sqrt{\frac{z-\sqrt{1-R^2}}{1-z}}dz.
\end{align*}

If we now deform the contour of integration $L_1$ onto the two-sided segment $[\sqrt{1-R^2},1]$ we obtain 
\begin{equation}\label{eq49}
 \int_{\sqrt{1-R^2}}^1q_{n}(x)[f(x)]^md\lambda_n(x)=0, \quad 0\leq m\leq n-\lfloor n/2\rfloor-1,
\end{equation}
where
\[
f(x)=\frac{R^2}{x^2-1}+1,\quad
d\lambda_n=(1-x)^{[n/2]}\sqrt{\frac{x^2-(1-R^2)}{1-x^2}} dx.
\]

Let $\alpha_1,\ldots,\alpha_N$ be the roots that the polynomial of real coefficients $q_{n}$ has in $(\sqrt{1-R^2},1)$. Since $f(x)$ is decreasing in $(\sqrt{1-R^2},1)$, we have
\[
\int_{\sqrt{1-R^2}}^1q_{n}(x)\prod_{k=1}^N[f(x)-f(\alpha_k)]d\lambda_n(x)\not=0,
\]
which in view of (\ref{eq49}) forces $N=n-\lfloor n/2\rfloor$.
\end{proof}

\begin{proof}[Proof of Theorem \ref{thm8}] The proof is based on applying Proposition \ref{prop2} to the number $\bar{\mu}:=(R-\sqrt{R^2-4})/2$. For this $\bar{\mu}$, we have that $L_r\subset G_{1/r}$ for all $\bar{\mu}<r<1$. For otherwise, there must exist $r_0\in(\bar{\mu},1)$ for which $L_{r_0}\cap L_{1/r_0}\not=\emptyset$. Hence we can find two numbers $w$ and $v$ such that $|w|=|v|=r_0$ and $\psi(w)=\psi(1/v)$. By (\ref{eqq11}), this implies that $1\geq  (|Rw|-1)\left(\left|R/v\right|-1\right)$, or equivalently, $r_0^2-Rr_0+1\geq 0$. This last inequality holds if and only if either $r_0\leq \bar{\mu}$ or $r_0\geq 1/\bar{\mu}$ $(>1)$, contradicting that $r_0\in(\bar{\mu},1)$.

Consider now the sequence of real numbers $\{w_n\}_{n=1}^\infty$ defined recursively as follows: $w_1$ is any number satisfying $\bar{\mu}< w_1<1$, and
\[
w_{n+1}=\frac{1}{R-w_{n}},\quad n\geq 1.
\]
It is easy to prove by induction that $w_n>\bar{\mu}$ for all $n\geq 1$, while straightforward computations yield that $\psi(w_{n+1})=\psi\left(1/\overline{w_n}\right)$, $n\geq 1$.

We can then invoke Proposition \ref{prop2} to conclude that $\mu=\bar{\mu}$ and that for each $p\geq 1$, $\Sigma_p$ consists of those points $z\in G_1$ for which the equation $z=Rw-1+(Rw-1)^{-1}$ has exactly $p$ solutions of largest modulus in $\mu<|w|<1$ (counting multiplicities), thereby establishing Theorem \ref{thm8}.
\end{proof}

\begin{proof}[Proof of Proposition \ref{prop8}] Part (a) will follow from Proposition \ref{prop1} once we prove that $\varphi(\psi(w))$ has a meromorphic continuation to $\mathfrak{D}:=\mathcal{\mathbb{D}}_1\setminus\{\mu\}$. For this, we shall prove that $\mathfrak{D}$ can be exhausted by continuously expanding domains $\mathfrak{D}_t$ satisfying the hypothesis of Proposition \ref{prop6} (to be precise, our domains differ from those contemplated in Proposition \ref{prop6} in that they expand as the parameter $t$ \emph{decreases} over an interval, but this, of course, is of no significance).

For $0< t\leq  \pi/2$, let us denote by $\mathfrak{D}_t$ the set
\begin{equation*}
\mathbb{D}_1\setminus \left\{w:\frac{2\cos t}{ R+\sqrt{R^2-4\cos^2t}}\leq |w|\leq \frac{2}{R+\cos t\sqrt{R^2-4}} ,\  |\arg w|\leq t \right\}.
\end{equation*}

Obviously, for $0< t_1<t_0\leq \pi/2$ (recall that $\rho=2/R$),
\[
\{w:\rho<|w|<1\}\subset \mathfrak{D}_{t_0}\subset \mathfrak{D}_{t_1}\subset \mathbb{D}_1,\quad \bigcap_{t<t_0}\mathfrak{D}_t = \overline{\mathfrak{D}}_{t_0}\setminus\mathbb{T}_1,
\]
and $\mathfrak{D}:=\bigcup_{0< t\leq \pi/2}\mathfrak{D}_{t}$. Then, by Proposition \ref{prop6}, in order to prove that $\varphi(\psi(w))$ has a meromorphic continuation to $\mathfrak{D}$, it suffices to show that
\begin{equation}\label{eqq30}
\psi(\mathfrak{D}_{\pi/2})\subset G_1,
\end{equation}
\begin{equation}\label{eqq10}
\psi(\overline{\mathfrak{D}_t})\subset \overline{G}_1\cup \psi(\mathfrak{D}^*_t),\quad 0< t< \frac{\pi}{2}.
\end{equation}
We only prove (\ref{eqq10}) as (\ref{eqq30}) is much easier and follows similarly.

Suppose, on the contrary, that (\ref{eqq10}) does not hold, that is, for some $0<t< \pi/2$, points $w_1=r_1e^{i\theta_1}\in \overline{\mathfrak{D}_t}$ and $w_2=r_2e^{i\theta_2}\in \mathbb{D}_1\setminus\mathfrak{D}_t$  can be found such that
\[
\psi(w_1)=\psi(1/\overline{w_2}).
\]
By (\ref{eqq11}), this equality holds if and only if $
w_1(R-\overline{w_2})=1$, which implies that
\begin{equation}\label{eqq19}
r_1(R-r_2)\leq 1,
\end{equation}
\begin{equation}\label{eqq16}
r_1r_2\cos\theta_2=Rr_1-\cos\theta_1,\quad
\sin\theta_1=-r_1r_2\sin\theta_2,
\end{equation}
\begin{equation}\label{eqq1}
2Rr_1\cos\theta_1=(R^2-r_2^2)r_1^2+1.
\end{equation}

Since $w_2\in \mathbb{D}_1\setminus\mathfrak{D}_t$, we have
\begin{equation}\label{eqq22}
\cos t\leq\cos \theta_2,
\end{equation}
and since $r_1\leq  1$ and $r_2<1$, we get from (\ref{eqq1}) that $\cos\theta_1>0$, which together with the second equality of (\ref{eqq16}) implies that
\begin{equation}\label{eqq25}
\cos \theta_2<\cos\theta_1 \quad \mathrm{unless}\quad \cos\theta_1=\cos\theta_2=1.
\end{equation}
Hence $\cos t<\cos\theta_1$, and given that $w_1\in \overline{\mathfrak{D}_t}$, we must have that either
\begin{equation*}
r_1\leq  \frac{2\cos t}{ R+\sqrt{R^2-4\cos^2t}}\leq r_2\quad \mathrm{or}\quad
r_2\leq \frac{2}{R+\cos t\sqrt{R^2-4}}\leq r_1.
\end{equation*}

In the first case, we obtain from  (\ref{eqq25}) and the first equality in (\ref{eqq16}) that
\[
\cos\theta_2< \frac{R}{r_1^{-1}+r_2}\leq \frac{R}{r_1^{-1}+r_1}\leq \cos t\quad \mathrm{if} \quad \cos\theta_2<\cos\theta_1,
\]
and that
\[
1=\frac{R}{r_1^{-1}+r_2}\leq \cos t\quad \mathrm{if} \quad \cos\theta_2=\cos\theta_1=1,
\]
contradicting (\ref{eqq22}) and the fact that $t>0$. In the second case, $r_1>2/(R+\sqrt{R^2-4})$ and $1<r_1(R-r_1)\leq r_1(R-r_2)$, contradicting (\ref{eqq19}).

It only remains to prove Proposition \ref{prop8}(b). Let us denote the meromorphic continuation of the map  $\varphi(z)$ to  $\overline{\mathbb{C}}\setminus \{R^2-2\}$  by the same letter $\varphi$, so that being  $\psi(w)$ meromorphic in $\overline{\mathbb{C}}$ and mapping $\overline{\mathbb{C}}\setminus\{\mu,1/\mu\}$ onto $\overline{\mathbb{C}}\setminus \{R^2-2\}$, we simply have that
\begin{equation}\label{eq121}
h_\varphi(w)=\varphi(\psi(w))=\frac{1}{\overline{\varphi(\psi(1/\overline{w}))}}, \quad w\in \overline{\mathbb{C}}\setminus\{\mu,1/\mu\}.
\end{equation}

Let $z\in\Sigma_0$ be fixed, and suppose that for some $u_1\in \mathbb{D}_1\setminus\{\mu\}$, $h_\varphi(u_1)=\varphi(z)$. Consider the sequence of points in the extended complex plane that is recursively generated out of $u_1$ by the relation $u_{n+1}=R-1/u_n$. It follows from (\ref{eqq11}) that $\psi(u_n)=\psi(1/\overline{u_{n+1}})$, so that by (\ref{eq121}), $h_\varphi(u_n)=\varphi(z)$ if $n$ is odd, while $h_\varphi(u_n)=1/\overline{\varphi(z)}$ if $n$ is even. As a consequence, not all the $u_n$'s can be contained in $\mathbb{D}_1$, for in such a case they would have to accumulate over $\mu$, but given that $\mu^2-R\mu+1=0$, we would also have that $|u_{n+1}-\mu|=|u_n-\mu|/|\mu u_n|>|u_n-\mu|$, yielding a contradiction.

Now, $u_1\in \mathbb{D}_1$, and if $u_n\in \mathbb{D}_1$ and $\psi(u_n)\not\in \overline{G}_1$, then  there is a unique $u'_n$ with $|u_n'|>1$ such that $\psi(u_n)=\psi(u_n')=\psi(1/\overline{u_{n+1}})$. Hence $u_{n+1}=1/\overline{u_n'}\in\mathbb{D}_1$. As a consequence, there must exist an index $N$ such that $u_{N}\in \mathbb{D}_1$ and $\psi(u_N)\in \overline{G}_1$, so that $|h_\varphi(u_N)|=|\varphi(\psi(u_N))|\leq 1$, which is only possible if $N$ is odd and $\psi(u_N)=z$.

Since the equation $\psi(w)=z$ has for roots $(z+2\pm\sqrt{z^2-4})/(2R)$, both lying in $|w|\leq \mu$, it follows that $u_1$ has to be an odd-indexed element of one of the two sequences generated out of these roots by the relation (\ref{eq122}), and it is easy to verify that all the odd-indexed elements of these two sequences are  solutions of the equation $h_\varphi(w)=\varphi(z)$.
\end{proof}

\nocite{*}
\bibliographystyle{cdraifplain}
\bibliography{xampl}

\begin{thebibliography}{99}

\bibitem{Andriev1}{V. Andrievskii, H.-P. Blatt, Erd\"{o}s-Tur\'{a}n type theorems on quasiconformal curves and arcs,
J. Approx. Theory, 97 (1999), 334-365.}

\bibitem{Andriev2}{V. Andrievskii, I. E. Pritsker, R. S. Varga, On zeros of polynomials orthogonal over a convex domain, Constr. Approx., 17 (2001), 209-225.}

\bibitem{Carleman}{T. Carleman, \"{U}ber die approximation analytischer funktionen durch lineare aggregate von vorgegebenen
potenzen,  Archiv. f\"{o}r Math. Atron. och Fysik, 17
(1922) 1-30.}

\bibitem{Davis}{P. J. Davis, The Schwarz function and its applications, The Carus Mathematical Monographs, 17, The Mathematical Association of America, Buffalo, N. Y., 1974.}

\bibitem{Eiermann}{M. Eiermann, H. Stahl, Zeros of orthogonal polynomials on regular $N$-gons, Lecture Notes in Mathematics, 1574 (1994), 187-189, Springer, Heidelberg.}


\bibitem{Gaier1}{D. Gaier, Lectures on complex approximation. Boston:
Birkh\"{a}user, 1987. Translated from German by Renate
McLaughlin.}

\bibitem{KS}{A. B. J. Kuijlaars, E. B. Saff, Asymptotic distribution of the zeros of Faber polynomials, Math. Proc. Camb.
Phil. Soc., 118 (1995), 437-447.}

\bibitem{Levin}{A. L. Levin, E. B. Saff, N. S. Stylianopoulos,
Zero distribution of Bergman orthogonal polynomials for certain
planar domains, Constr. Approx., 19 (2003), 411-435.}

\bibitem{victor}{V. Maymeskul, E. B. Saff, Zeros of polynomials orthogonal over
regular $N$-gons, J. Approx. Theory, 122 (2003), 129-140.}

\bibitem{mina1}{E. Mi\~{n}a-D\'{\i}az, An asymptotic integral representation
for Carleman orthogonal polynomials, Int. Math. Res. Notices,  2008 (2008), article ID rnn065, 38 pages.}

\bibitem{mina2}{E. Mi\~{n}a-D\'{\i}az,  E. B. Saff, N. S. Stylianopoulos, Zero distributions
for polynomials orthogonal with weights over
certain planar regions, Comput. Methods Funct. Theory, 5 (2005), 185-221.}

\bibitem{Nehari}{Z. Nehari, Conformal mapping, McGraw-Hill Book Company, Inc., 1952.}

\bibitem{Saff}{E. B. Saff, V. Totik, Logarithmic potentials with external fields,
Berlin: Springer-Verlag, 1997.}

\bibitem{Suetin1}{P. K. Suetin, Polynomials orthogonal over a region and Bieberbach polynomials,
Proceedings of the Steklov Institute of Mathematics, Amer.
Math. Soc., Providence, Rhode Island, 1975.}

\bibitem{Tsuji}{M. Tsuji, Potential theory in modern function theory, Chelsea, N. Y., 1959.}

\bibitem{Ullman2}{J. L. Ullman, Studies in Faber polynomials I, Trans. Amer. Math. Soc., 73 (1960), 515-528.}

\bibitem{Ullman1}{J. L. Ullman, Toeplitz matrices associated with a semi-infinite Laurent series,  Proc. London Math. Soc., 22  (1971), 164-192.}
\end{thebibliography}
\end{document}